\pdfoutput=1
\documentclass[a4paper]{amsart}

\newif\ifspringer
\springerfalse

\usepackage[hidelinks]{hyperref}
\usepackage{microtype}

\ifspringer
  \usepackage{amsmath,amssymb,mathptmx}
\else
  \usepackage{amsmath, amssymb, amsthm}
\fi
\usepackage{tikz-cd}
\usepackage[T1]{fontenc}
\usepackage{spectralsequences}
\usepackage{cleveref}
\usepackage{rotating}
\usepackage{listings}
\usepackage{float}
\usepackage{graphicx}

\ifspringer
  \smartqed
\else
  
  \usepackage[doi=false,url=false,isbn=false,citestyle=authoryear,style=alphabetic]{biblatex}
  \AtEveryBibitem{\clearfield{note}}
  \newtheorem{theorem}{Theorem}[section]
  \newtheorem{lemma}[theorem]{Lemma}
  \newtheorem{corollary}[theorem]{Corollary}

  \theoremstyle{definition}
  \newtheorem{definition}[theorem]{Definition}
  \newtheorem{example}[theorem]{Example}

  \newtheorem*{remark}{Remark}
\fi

\DeclareMathOperator\Spec{Spec}
\DeclareMathOperator\Ext{Ext}

\DeclareMathOperator\Mod{Mod}
\DeclareMathOperator\Ind{Ind}

\DeclareMathOperator\tr{tr}

\newcommand\Z{{\mathbb{Z}}}

\newcommand\F{{\mathbb{F}}}

\renewcommand\P{{\mathbb{P}}}
\newcommand\RP{{\mathbb{RP}}}
\renewcommand\d{\mathrm{d}}

\newcommand\E{{\mathcal{E}}}
\renewcommand\O{{\mathcal{O}}}
\newcommand\M{{\mathcal{M}}}
\newcommand\Ell{{\mathcal{E}ll}}

\newcommand\tmf{{\mathrm{tmf}}}
\newcommand\TMF{{\mathrm{TMF}}}
\newcommand\KO{{\mathrm{KO}}}
\newcommand\ko{{\mathrm{ko}}}
\newcommand\kappabar{{\bar{\kappa}}}

\newcommand\Sp{{\mathrm{Sp}}}
\DeclareMathOperator\QCoh{QCoh}

\DeclareMathOperator\RO{RO}
\DeclareMathOperator\Sq{Sq}
\DeclareMathOperator*\colim{colim}

\newcommand\A{{\mathcal{A}}}
\newcommand\aesq[1]{{\A^{C_2} \square_{\A^{C_2}(#1)}}}
\newcommand\HF{{H\underline{\F}}}
\newcommand\Mt{{\mathbb{M}_2}}
\newcommand\bxi{{\bar{\xi}}}
\newcommand\Y{Y}

\SseqErrorToWarning{range-overflow}

\sseqset{
  hzero/.sseq style = {},
  hone/.sseq style = {},
  htwo/.sseq style = {},
}

\DeclareSseqGroup\hzerotower {} {
  \class(0, 0)
  \savestack
  \DoUntilOutOfBounds {
    \class(\lastx, \lasty + 1)
    \structline[hzero]
  }
  \restorestack
}

\DeclareSseqGroup\hzerotowereta {} {
  \hzerotower
  \class(1, 1) \structline[hone]
  \class(2, 2) \structline[hone]
}

\tikzset{circ/.style = {fill, circle, inner sep = 0, minimum size = 3}}

\newcounter{eta}
\newcounter{nu}
\newenvironment{celldiagram}{
  \setcounter{eta}{0}
  \setcounter{nu}{0}
  \def\n##1{\node [circ] at (0, ##1) {};}
  \def\invert##1{\setcounter{##1}{1 - \the\value{##1}}}
  \def\getside##1{\ifnum\the\value{##1}=0 right\else left\fi}
  \def\two##1{\draw [blue, thick] (0, ##1) to (0, ##1 + 1);}
  \def\eta##1{\draw [red, thick] (0, ##1) to [bend \getside{eta}=50] (0, ##1 + 2);\invert{eta}}
  \def\nu##1{\draw [green!50!black, thick] (0, ##1) to [bend \getside{nu}=50] (0, ##1 + 4);\invert{nu}}
}{}

\NewSseqGroup\hook {} {
  \class(1, 1)
  \class(2, 2) \structline[hone]
  \class(4, 2) \structline[htwo](4, 2)(1, 1)
  \class(7, 3) \structline[htwo]
}
\NewSseqGroup\htwoedge {} {
  \class(4, 2)
  \class(7, 3) \structline[htwo]
}

\title{\texorpdfstring{$C_2$}{C2}-equivariant topological modular forms}
\author{Dexter Chua}
\ifspringer
  \institute{Dexter Chua\at Department of Mathematics, Harvard University, Cambridge, MA 02138, \\\email{dexter@math.harvard.edu}}
\fi
\def\labelscalefactor{0.9}

\sseqset{Adams grading, classes = { fill }, class label handler = { \def\result{\scalebox{\labelscalefactor}{$#1$}} }, grid = go, class labels = { left = 0cm }, differentials = {blue}}
\sseqset{small/.global sseq style = {
    x tick step = 2,
    scale = {\ifspringer 0.339 \else 0.354 \fi}
  },
  large/.global sseq style = {
    x tick step = 8,
    y tick step = 8,
    grid step = 4,
    classes = { minimum size = 1.2 },
    scale = {\ifspringer 0.16 \else 0.19\fi},
  }
}
\begin{document}
\begin{abstract}
  We compute the homotopy groups of the $C_2$ fixed points of equivariant topological modular forms at the prime $2$ using the descent spectral sequence. We then show that as a $\TMF$-module, it is isomorphic to the tensor product of $\TMF$ with an explicit finite cell complex.
  \ifspringer
  \fi
\end{abstract}
\maketitle
\tableofcontents
\section{Introduction}
Topological $K$-theory is one of the first examples of generalized cohomology theories. It admits a natural equivariant analogue --- for a compact Hausdorff $G$-space $X$, the group $\KO^0_G(X)$ is the Grothendieck group of $G$-equivariant vector bundles over $X$. In particular, $\KO^0_G(*) = \operatorname{Rep}(G)$ is the representation ring of $G$.

As in the case of non-equivariant $K$-theory, this extends to a $G$-equivariant cohomology theory $\KO_G$, and is represented by a genuine $G$-spectrum. We shall call this $G$-spectrum $\KO$, omitting the subscript, as we prefer to think of this as a global equivariant spectrum --- one defined for all compact Lie groups. The $G$-fixed points of this, written $\KO^{\mathcal{B}G}$, are a spectrum analogue of the representation ring, with $\pi_0 \KO^{\mathcal{B}G} = \KO^0_G(*) = \operatorname{Rep}(G)$ (more generally, $\pi_n \KO^{\mathcal{B}G} = \KO^{-n}_G(*)$).

These fixed point spectra are readily computable as $\KO$-modules. For example,
\[
  \KO^{\mathcal{B}C_2} = \KO \vee \KO,\quad \KO^{\mathcal{B} C_3} = \KO \vee \mathrm{KU}.
\]
This corresponds to the fact that $C_2$ has two real characters, while $C_3$ has a real character plus a complex conjugate pair. If one insists, one can write $\mathrm{KU} = \KO \otimes C\eta$, providing an arguably more explicit description of $\KO^{\mathcal{B} C_3}$ as a $\KO$-module. In general, $\KO^{\mathcal{B}G}$ decomposes as a direct sum of copies of $\KO$, $\mathrm{KU}$ and $\mathrm{KSp}$, with the factors determined by the representation theory of $G$ \cite[p.133--134]{segal-equivariant-k-theory}.

From the chromatic point of view, the natural object to study after $K$-theory is elliptic cohomology, or its universal version, topological modular forms. Equivariant elliptic cohomology, in various incarnations, has been of interest to many people, include geometric representation theorist and quantum field theorists. Most recently, in \cite{equivariant-tmf}, Gepner and Meier constructed integral equivariant elliptic cohomology and topological modular forms for compact abelian Lie groups, following the outline in \cite{elliptic-survey} and the groundwork in \cite{elliptic-i,elliptic-ii,elliptic-iii}. The introduction in \cite{equivariant-tmf} provides a nice overview of the relevant history, whose efforts we shall not attempt to reproduce.

The spectra $\TMF^{\mathcal{B}C_n}$ can be constructed as follows: in \cite{elliptic-ii}, Lurie constructed the universal oriented\footnote{``oriented'' refers to complex orientation of the associated cohomology theory} (spectral) elliptic curve, which we shall denote $p\colon \E \to \M$. Equivariant $\TMF$ is then constructed with the property that
\[
  \TMF^{\mathcal{B}C_n} = \Gamma(\E[n]; \O_{\E[n]}),\quad \TMF^{\mathcal{B}S^1} = \Gamma(\E; \O_\E),
\]
where $\E[n]$ are the $n$-torsion points of the elliptic curve. This is to be compared to the homotopy fixed points (with trivial group action), where $\E$ is replaced by the formal group $\hat{\E}$.

We are interested in explicit descriptions of these spectra as $\TMF$-modules. Much work was done by Gepner--Meier themselves: in \cite[Theorem 1.1]{equivariant-tmf}, they computed 
\[
  \TMF^{\mathcal{B}S^1} = \TMF \oplus \Sigma \TMF.
\]
This corresponds to the fact that the coherent cohomology of a (classical) elliptic curve is concentrated in degrees $0$ and $1$ by Serre duality.

As for finite groups, \cite[Example 9.4]{equivariant-tmf} argues that if $\ell \nmid |G|$ or $\ell > 3$, then $\TMF^{\mathcal{B}G}_{\ell}$ splits as sums of shifts of $\TMF_1(3)$, $\TMF_1(2)$ and $\TMF$. Further, $\TMF_1(3)$ and $\TMF_1(2)$ can themselves be described as the smash product of $\TMF$ with an 8- and 3-cell complex respectively (see \cite[Section 4]{homology-tmf} for details). Thus, we have an explicit description of $\TMF^{\mathcal{B}G}_\ell$ as a $\TMF_\ell$-module.

This leaves us with the case where $\ell = 2, 3$ and $\ell \mid |G|$. In this paper, we compute $\TMF^{\mathcal{B}C_2}$ at the prime $2$.

\begin{theorem}\label{thm:main}
  There is a (non-canonical) isomorphism of $2$-completed $\TMF$-modules
  \[
    \TMF^{\mathcal{B}C_2} \cong \TMF \oplus \TMF \oplus \TMF \otimes DL,
  \]
  where $DL$ is the spectrum $S^{-8} \cup_{\nu} S^{-4} \cup_{\eta} S^{-2} \cup_2 S^{-1}$, as depicted in \Cref{fig:cell-dl}.
\end{theorem}
This space $DL$ is so named because its dual $L$ is a split summand of the spectrum $L_0$ defined in \cite[Definition 2.3]{tmf-tate}; in fact, $L_0 = L \oplus S^0$.
\begin{figure}[ht]
  \centering
  \begin{tikzpicture}[scale=0.5]
    \begin{celldiagram}
      \two{-2}
      \eta{-4}
      \nu{-8}
      \n{-8} \n{-4} \n{-2} \n{-1}

      \foreach \y in {1,2,4,8} {
        \node [left] at (-0.8, -\y) {$-\y$};
      }
    \end{celldiagram}
    \node [right] at (1, -6) {$\nu$};
    \node [right] at (0.5, -3) {$\eta$};
    \node [right] at (0, -1.5) {$2$};
  \end{tikzpicture}
  \caption{Cell diagram of $DL$}\label{fig:cell-dl}
\end{figure}
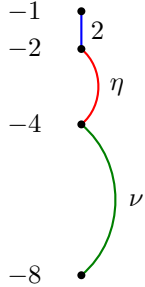

\begin{remark}
  Despite being a $4$-cell complex, the $\TMF$-module $\TMF \otimes DL$ is really a rank-$2$ $\TMF$-module. We claim that after base change to the flat cover $\TMF_1(3) = v_2^{-1}\mathrm{BP}\langle 2 \rangle$, the module $\TMF_1(3) \otimes DL$ is free of rank $2$.

  We first observe what happens after base change to $\tmf_1(3) = \mathrm{BP}\langle 2 \rangle$. In $\pi_* \tmf_1(3)$, the elements $\eta$ and $\nu$ are killed. We claim that the cell diagram of $\tmf_1(3) \otimes DL$ as a $\tmf_1(3)$-module is what is given in \Cref{fig:cell-dl-tmf13}.

  \begin{figure}[H]
    \centering
    \begin{tikzpicture}[scale=0.5]
      \begin{celldiagram}
        \two{-2}
        \draw (0, -1) edge [bend left=50, red, thick] (0, -4);
        \draw (0, -1) edge [bend left=50, green!50!black, thick] (0, -8);
        \n{-8} \n{-4} \n{-2} \n{-1}

        \foreach \y in {1,2,4,8} {
          \node [left] at (-0.8, -\y) {$-\y$};
        }
      \end{celldiagram}
      \node [left] at (1.5, -4.5) {$v_2$};
      \node [left] at (0.7, -2.5) {$v_1$};
      \node [left] at (0, -1.5) {$2$};
    \end{tikzpicture}
    \caption{Cell diagram of $DL$ over $\tmf_1(3)$}\label{fig:cell-dl-tmf13}
  \end{figure}
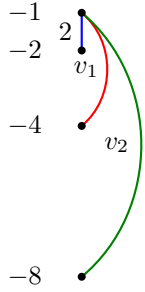

  To see this, recall that $v_n$ is a filtration one element detected by $Q_n$. We can expand
  \[
    Q_1 = \Sq^2 \Sq^1 + \Sq^1 \Sq^2,\quad Q_2 = \Sq^1 \Sq^2 \Sq^4 + \Sq^5 \Sq^2 + \Sq^6 \Sq^1 + \Sq^4 \Sq^2 \Sq^1.
  \]
  So $Q_1$ and $Q_2$ send the $(-4)$- and $(-8)$-cell to the top cell respectively. Since there are no other possible attaching maps for degree reasons ($v_n$ is only well-defined up to higher filtration), the cell diagram must be as described.

  In $\TMF_1(3)$, we further invert $v_2$, and the top and bottom cell cancel out. So we are left with two free cells in degree $-2$ and $-4$.
\end{remark}

\begin{remark}
  While the theorem is stated for $\TMF$, the same result holds for any elliptic cohomology theories in the sense of \Cref{section:equivariant}. Indeed, by \cite[Theorem 7.2]{affineness-chromatic}, taking global sections gives an equivalence of $\infty$-categories
  \[
    \Gamma\colon \QCoh(\M) \overset\sim\to \Mod_\TMF.
  \]
  Thus, as quasi-coherent sheaves, we have
  \[
    p_* \mathcal{O}_{\E[2]} \cong \mathcal{O}_\M \oplus \mathcal{O}_\M \oplus \mathcal{O}_\M \otimes DL.
  \]
  By the universality of $\M$, the same must hold for all other elliptic cohomology theories.
\end{remark}

\subsection{Outline of proof}
To prove the theorem, we begin by computing the homotopy groups of $\TMF^{\mathcal{B}C_2}$. As in the case of $\TMF$, there is a descent spectral sequence computing $\pi_* \TMF^{\mathcal{B} C_2}$ whose $E_2$ page is the coherent cohomology of the $2$-torsion points of the (classical) universal elliptic curve.

Upon computing the $E_2$ page for $\TMF^{\mathcal{B}C_2}$, one immediately observes that there are two copies of $\TMF$'s $E_2$ page as direct summands (as one would expect from the answer). We can identify these copies as follows:
\begin{enumerate}
  \item Applying $\Gamma$ to the map $p\colon \E[2] \to \M$ induces $1 \colon \Gamma(\M; \O_\M) \to \Gamma(\E[2]; \O_{\E[2]})$. This is split by the identity section.
  \item Since $\TMF$ is a genuine $C_2$-equivariant cohomology theory, we get a norm map $\TMF_{h C_2} = \TMF \otimes \RP^\infty_+ \to \TMF^{\mathcal{B}C_2}$. Restricting to the bottom cell of $\RP^\infty_+$ gives us a transfer map $\tr\colon \TMF \to \TMF^{\mathcal{B}C_2}$.
\end{enumerate}
We will explore these further in \Cref{subsection:unit-transfer}.

To simplify the calculation, we can quotient out these factors, and rephrase our original theorem as
\begin{theorem}
  There is an isomorphism
  \[
    \overline{\TMF^{\mathcal{B}C_2}} \equiv \TMF^{\mathcal{B}C_2} / (1, \tr) \simeq \TMF \otimes DL.
  \]
\end{theorem}
This is proven by computing the homotopy groups of $\overline{\TMF^{\mathcal{B} C_2}}$ via its descent spectral sequence, which is now reasonably sparse, followed by an obstruction theory argument. This implies the original theorem via the observation
\begin{lemma}
  Any cofiber sequence of $\TMF$-modules
  \[
    \TMF \oplus \TMF \to {?} \to \TMF \otimes DL
  \]
  splits.
\end{lemma}

\begin{proof}
  We have to show that
  \[
    [\TMF \otimes DL, \Sigma \TMF \oplus \Sigma \TMF]_\TMF = 0.
  \]
  This is equivalent to showing that $\pi_{-1} \TMF \otimes L = 0$. This follows immediately by running the long exact sequences building $\TMF \otimes L$ from its cells, since $\pi_{-2} \TMF = \pi_{-3} \TMF = \pi_{-5} \TMF = \pi_{-9} \TMF = 0$.
\end{proof}

\begin{remark}
  At first I only computed the homotopy groups of $\TMF^{\mathcal{B}C_2}$. The above identification was discovered when I, for somewhat independent reasons, looked into the homotopy groups of $\TMF \otimes L$, and observed that they looked almost the same as that of $\TMF^{\mathcal{B}C_2}$. It is, however, $\TMF \otimes DL$ that shows up above; there is a cofiber sequence
  \[
    \TMF \otimes L \to \TMF \otimes DL \to \KO,
  \]
  which induces a \emph{short} exact sequence in homotopy groups. Thus, the homotopy groups of $\TMF \otimes DL$ and $\TMF \otimes L$ differ by a single copy of $\pi_* \KO$, which is hard to notice after inverting $\Delta$. On the other hand, the corresponding classes have different Adams--Novikov filtrations, which makes them easy to distinguish in practice.
\end{remark}

\begin{remark}
  As part of the proof, we compute the homotopy groups $\pi_* \TMF^{\mathcal{B}C_2}$. To describe the group explicitly, under the decomposition, it remains to specify $\pi_* \TMF \otimes DL = \pi_* \overline{\TMF^{\mathcal{B} C_2}}$. This group is given by the direct sum of the $\ko$-like parts, namely
  \[
    \bigoplus_{k \in \Z} \pi_*\Sigma^{8 + 24k} \ko \oplus \pi_*\Sigma^{16 + 24k} \ko,
  \]
  and what is depicted in \Cref{fig:anss-e9-page,fig:anss-e-infty-page}. In these figures, each dot is a copy of $\Z/2$, and the greyed out classes are ones that do not survive the spectral sequence (that is, the homotopy groups are given by the black dots). This part is $192$-periodic via $\Delta^8$-multiplication.
\end{remark}

\subsection{Overview}
In \Cref{section:equivariant} we provide relevant background on equivariant elliptic cohomology. Building upon the results in \cite{equivariant-tmf}, we construct $C_n$-equivariant elliptic cohomology as a functor $\Sp_{C_n}^{\mathrm{op}} \to \QCoh(\E[n])$, which gives us the transfer map $\tr$. We then provide an explicit description of the descent spectral sequence for quasi-coherent sheaves over $\M$.

In \Cref{section:e2}, we compute the Hopf algebroid presenting $\E[2]$ and subsequently the $E_2$ page of the descent spectral sequence for $\overline{\TMF^{\mathcal{B}C_2}}$ using the $2$-Bockstein spectral sequence. Unfortunately, the coaction involves division in a fairly complex ring, and cocycle manipulations throughout the paper are performed with the aid of \texttt{sage}.

In \Cref{section:differentials}, we compute the differentials in the descent spectral sequence. The key input here is the fact that there is a norm map $\TMF_{h C_2} \to \TMF^{\mathcal{B}C_2}$ whose composite all the way down to $\TMF^{h C_2}$ is well-understood in terms of stunted projective spaces. This provides us with a few permanent classes, which combined with the $\TMF$-module structure lets us compute all the differentials. Our calculations will make heavy use of synthetic spectra \cite{synthetic}, whose relation to the Adams spectral sequence is laid out in \cite[Section 9]{manifold-synthetic}.

In \Cref{section:ident}, we conclude the story by constructing a map $\TMF \otimes DL \to \overline{\TMF^{\mathcal{B}C_2}}$ via obstruction theory and showing that it is an isomorphism.

In \Cref{section:connective}, we use entirely different methods to study properties of a hypothetical connective version of $C_2$-equivariant $\TMF$, which we call $\tmf_{C_2}$ (we put the $C_2$ subscript since we do not purport to describe a global equivariant $\tmf$). We shall show that under reasonable assumptions, $\Delta^{-1} (\tmf_{C_2})^{\mathcal{B}C_2}$ is dual to $\TMF^{\mathcal{B} C_2}$ in the category of $\TMF$-modules.

In \Cref{section:stunted}, we prove some basic properties of the norm map that we use in \Cref{section:differentials}.

Finally, in \Cref{section:sage}, we include the \texttt{sage} code we used for our cocycle manipulations.

\subsection{Conventions}
\begin{itemize}
  \item All categories are $\infty$-categories.
  \item We use $\otimes$ to denote the smash product of spectra.
  \item Unless otherwise specified, we work in the category of $\TMF$-modules, and all maps are $\TMF$-module maps. Further, we implicitly complete at the prime $2$.
  \item Our charts follow the same conventions as, say, \cite{tilman-tmf}. In each bidegree, a solid round dot denotes a copy of $\Z/2$. More generally, $n$ concentric circles denotes a copy of $\Z/2^n$. A white square denotes $\Z$. A line of slope $1$ denotes $h_1$ multiplication and a line of slope $\frac{1}{3}$ denotes $h_2$ multiplication. An arrow with a negative slope denotes a differential. Dashed lines denote hidden extensions. In particular, a dashed vertical line is a hidden $2$-extension. We use Adams grading, so that the horizontal axis is $t - s$ and vertical axis is $s$.
  \item All synthetic spectra will be based on $BP$. We choose our grading conventions so that $\pi_{t - s, s}(\nu X/\tau) = \Ext_{E_* E}^{s, t}(E_*, E_* X)$, i.e.\ $\pi_{x, y}$ shows up at coordinates $(t - s, s) = (x, y)$ in an Adams chart. Under these grading conventions, $\tau$ has bidegree $(0, -1)$. This is not the grading convention used by \cite{synthetic} and \cite{manifold-synthetic}; $\mathbb{S}^{a, b}$ in their grading is $S^{a, b - a}$ in ours.
  \item To avoid confusing the synthetic analogue functor $\nu$ with the element $\nu$ in the homotopy groups of spheres, we always write the former as $\nu(X)$ with the brackets.
  \item If $R$ is a (discrete) ring and $\alpha \in R$, we write $O(\alpha)$ for an unspecified element that is $\alpha$-divisible. For example, if $f = g + O(2)$, this means $f$ and $g$ agree mod $2$.
\end{itemize}

\ifspringer
  \begin{acknowledgements}
\else
  \subsection{Acknowledgements}
\fi
I would like to thank Robert Burklund for helpful discussions on various homotopy-theoretic calculations, especially regarding the application of synthetic spectra in \Cref{section:differentials}. Further, I benefited from many helpful discussions with Sanath Devalapurkar, Jeremy Hahn, and Lennart Meier regarding equivariant $\TMF$ and equivariant homotopy theory in general. Robert, Lennart and an anonymous referee also provided many helpful comments on an earlier draft. Finally, the paper would not have been possible without the support of my advisor, Michael Hopkins, who suggested the problem and provided useful guidance and suggestions throughout.

The author was partially supported by NSF grants DMS-1803766 and DMS-1810917 through his advisor.

\ifspringer
  \end{acknowledgements}
\fi

\section{Equivariant elliptic cohomology}\label{section:equivariant}
\subsection{Elliptic cohomology}
The starting point of equivariant elliptic cohomology is the notion of an oriented (spectral) elliptic curve, which was introduced by Lurie in \cite[Section 2]{elliptic-i} and \cite[Section 7.2]{elliptic-ii}. We should think of this as a spectral version of an elliptic curve, accompanied with a complex orientation of the associated cohomology theory.

Let $X$ be a non-connective spectral Deligne--Mumford stack, and $p\colon E \to X$ an oriented elliptic curve. In \cite[Construction 5.4, Proposition 8.2]{equivariant-tmf}, Gepner--Meier constructs an $S^1$-equivariant elliptic cohomology functor
\[
  \Ell_{S^1}\colon \Sp_{S^1}^{\mathrm{op}} \to \QCoh(E),
\]
which is a limit-preserving symmetric monoidal functor satisfying
\[
  \Ell_{S^1} ((S^1 / C_m)_+) = \mathcal{O}_{E[m]}.
\]

We begin by extending this to a functor on $\Sp_{C_n}$.

\begin{lemma}
  There is an elliptic cohomology functor
  \[
    \Ell_{C_n}^E\colon \Sp_{C_n}^{\mathrm{op}} \to \QCoh(E[n])
  \]
  such that for any $m \mid n$, we have a natural identification
  \[
    \Ell_{C_n}^E(({C_n} / C_m)_+) = \O_{E[m]},
  \]
  where we identify $\O_{E[m]}$ with its direct image in the $\QCoh(E[n])$ under the inclusion.

  Moreover, if $f: X' \to X$ is a morphism almost of finite presentation, then
  \[
    f^* \Ell_{C_n}^E(X) = \Ell_{C_n}^{f^* E}(X) \in \QCoh(f^* E[n]).
  \]
\end{lemma}
If there is no risk of confusion, we omit the superscript ${}^E$.

\begin{proof}
  Let $\Ind_{C_n}^{S^1}\colon \Sp_{C_n} \to \Sp_{S^1}$ be the induction map, left adjoint to the restriction map. Then $\Ind_{C_n}^{S^1}((C_n / C_m)_+) = (S^1 / C_m)_+$.

  Since the restriction map is symmetric monoidal under the smash product, $\Ind_{C_n}^{S^1}$ is oplax monoidal. Thus, the composite
  \[
    \begin{tikzcd}
      \Ell_{C_n}^*\colon \Sp_{C_n}^{\mathrm{op}} \ar[r, "\Ind_{C_n}^{S^1}"] & \Sp_{S^1}^{\mathrm{op}} \ar[r, "\Ell_{S^1}"] & \QCoh(E)
    \end{tikzcd},
  \]
  is lax monoidal. Since $S^0$ is a coalgebra in $\Sp_{C_n}$ and every object in $\Sp_{C_n}$ is naturally an $S^0$-comodule, it follows that this functor canonically factors through the category of $\Ell_{C_n}^*(S^0) = \O_{E[n]}$-modules in $\QCoh(E)$, which is equivalent to $\QCoh(E[n])$.\footnote{One has to check that the ring structure on $\O_{E[n]} = \Ell_{C_n}^*(S^0)$ that arises this way is the standard ring structure, which follows from the construction of $\Ell_{S^1}$.}

  Functoriality in $X$ follows from functoriality in the $S^1$ case as in \cite[Proposition 5.6]{equivariant-tmf}.
\end{proof}

\begin{remark}
  Unlike the case of $S^1$, the map $\Ell_{C_n}^E \colon \Sp_{C_n}^{\mathrm{op}} \to \QCoh(E[n])$ is in general not symmetric monoidal.
\end{remark}

\begin{corollary}
  There is a $C_n$-spectrum $R$ such that for any $C_n$-spectrum $Z$, we have
  \[
    (R^Z)^{C_n} = \Gamma(E[n], \Ell_{C_n}(Z)).
  \]
\end{corollary}
We call this $R$ the $C_n$-spectrum associated to the elliptic curve $E \to X$. For example, when $E \to X$ is the universal elliptic curve, then $R = \TMF$.

This follows the argument of \cite[Construction 8.3]{equivariant-tmf}.
\begin{proof}
  By the spectral Yoneda's lemma \cite[Proposition 4.8.2.18]{ha}, the Yoneda embedding $\Sp_{C_n} \to \operatorname{Fun}^R(\Sp_{C_n}^{\mathrm{op}}, \Sp)$ is an equivalence. Since $\Sp_{C_n}$ is presentable, by \cite[Corollary 5.5.2.9(1)]{htt}, a functor $\Sp_{C_n}^{\mathrm{op}} \to \Sp$ is a right adjoint iff its opposite preserves colimits, i.e.\ it preserves limits. Thus, we have to show that the functor $Z \mapsto \Gamma(E[n], \Ell_{C_n}(Z))$ preserves limits as a functor $\Sp_{C_n}^{\mathrm{op}} \to \Sp$.

  \begin{itemize}
    \item By construction $\Ell_{S^1}: \Sp_{S^1}^{\mathrm{op}} \to \QCoh(E)$ preserves limits.
    \item Since $\Ind_{C_n}^{S^1} \colon \Sp_{C_n} \to \Sp_{S^1}$ is a left adjoint, it preserves colimits, hence its opposite preserves limits. So $\Ell_{C_n}^*\colon \Sp_{C_n}^{\mathrm{op}} \to \QCoh(E)$ preserves limits.
    \item Since $\QCoh(E[n])$ is the category of $\O_{E[n]}$-modules in $\QCoh(E)$, the forgetful functor $\QCoh(E[n]) \to \QCoh(E)$ creates limits. So $\Ell_{C_n}\colon \Sp_{C_n}^{\mathrm{op}} \to \QCoh(E[n])$ preserves limits.
    \item Finally, $\Gamma \colon \QCoh(E[n]) \to \Sp$ is a right adjoint and preserves limits.
  \end{itemize}
\end{proof}

We are interested in these global sections, which we can write as
\[
  \Gamma(E[n]; \Ell_{C_n}(Z)) = \Gamma(X; p_* \Ell_{C_n}(Z)).
\]
By computing $p_* \Ell_{C_n}(S^0)$, this lets us understand the global sections in terms of quasi-coherent sheaves on $X$ itself. This pushforward is fairly nice by virtue of
\begin{lemma}
  The map $[n]\colon E \to E$ is flat, hence so is $p\colon E[n] \to X$.
\end{lemma}

\begin{proof}
  To check that $[n]\colon E \to E$ is flat, observe that by \cite[Theorem 2.3.1]{katz-mazur}, the map on underlying (classical) stacks is flat. The condition that $[n]^* \pi_t \O_E = \pi_t \O_E$ as sheaves on the underlying stack is automatic, since $\pi_t \O_E = p^* \pi_t \O_X$ and $p[n] = p$.

  For the second part, we have a pullback square
  \[
    \begin{tikzcd}
      E[n] \ar[d, "p"] \ar[r] & E \ar[d, "{[n]}"] \\
      X \ar[r] & E
    \end{tikzcd}
  \]
  where the bottom map is the identity section, and flat morphisms are closed under pullbacks.
\end{proof}

\begin{corollary}[{\cite[Lemma 8.1]{equivariant-tmf}}]\label[corollary]{cor:underlying-torsion}
  The underlying stack of $E[n]$ are the $n$-torsion points of the underlying stack of $E$.
\end{corollary}

\begin{proof}
  More generally, given a pullback of a flat morphism between non-connective spectral Deligne--Mumford stacks, it is also a pullback on the underlying classical stacks. To see this, since being flat and a pullback is local, we may assume that the stacks are in fact affine, in which case the result is clear.
\end{proof}

\subsection{The unit and transfer maps}\label{subsection:unit-transfer}
There are two natural maps
\[
  1, \tr\colon \O_X \to p_* \O_{E[n]}.
\]

The map $1$ is adjoint to the identity map $p^* \mathcal{O}_X = \mathcal{O}_{E[n]} \to \mathcal{O}_{E[n]}$, and is a map of $\O_X$-algebras. In particular, it is an $\O_X$-module homomorphism that sends $1$ to $1$. If $X$ were affine, then this comes from taking the global sections of $p\colon E[n] \to X$. This map is split by the identity section $X \to E[n]$.

The trace map $\tr$ comes from stable equivariant homotopy theory itself. To avoid double subscripts, set $G = C_n$. In the category of $G$-spectra, there are maps
\[
  G_+ \to S^0 \to G_+
\]
whose composition is $\sum_{g \in G} g$. The first map comes from applying $\Sigma^\infty_+$ to the map of unbased $G$-spaces $G \to *$, whereas the second map is the Spanier--Whitehead dual of the first map, using the self-duality of $G_+$. Informally, it sends $1 \mapsto \sum_{g \in G} G$.

Now $\Ell_G(G_+) = \Ell_{\{e\}}(S^0) = \O_X$, so we get maps of $\O_{E[n]}$-modules
\[
  \O_X \overset\tr\to \O_{E[n]} \to \O_X
\]
whose composite is $n$ (since $G$ acts trivially on $\O_X$). Applying $p_*$, we get a map $\tr\colon \O_X \to p_* \O_{E[n]}$.

It will be useful to relate this to the norm map of $C_n$-spectra. Let $R$ be the $C_n$-spectrum associated to $E$. Then unwrapping the definitions, we see that $\tr$ is the $C_n$-fixed points of the map
\[
  R \otimes G_+ \longrightarrow R \otimes S^0,
\]
obtained by tensoring up the unique map $G \to *$. Similarly, the norm map is induced by
\[
  R \otimes EG_+ \longrightarrow R \otimes S^0,
\]
using the Adams isomorphism $(R \otimes EG_+)^{\mathcal{B}G} = R_{hG}$. Since $G$ includes into $EG$, the trace map factors as
\[
  \tr\colon R \longrightarrow R_{h G} \overset{\mathrm{Nm}}\longrightarrow R^{\mathcal{B}G},
\]
where the left-hand map is the usual inclusion. Since $G$ acts trivially on the underlying spectrum $R$, we have $R_{h G} = R \otimes BG_+$, and the left-hand map is the inclusion of the bottom cell of $BG_+$.

We now define $\overline{p_* \O_{E[n]}}$ by the following cofiber sequence in $\QCoh(X)$:
\[
  \O_X \oplus \O_X \overset{1 \oplus \tr}\longrightarrow p_*\O_{E[n]} \longrightarrow \overline{p_* \O_{E[n]}}.
\]
We then write
\[
  \begin{aligned}
    R^{\mathcal{B}C_n} &= \Gamma(X; p_* \O_{\E[n]}),\\
    \overline{R^{\mathcal{B}C_n}} &= \Gamma(X; \overline{p_* \O_{\E[n]}}).
  \end{aligned}
\]
In particular, when $X = \M$ and $R = \TMF$, we have a cofiber sequence
\[
  \TMF \oplus \TMF \overset{1 \oplus \tr}\longrightarrow \TMF^{\mathcal{B}C_n} \longrightarrow \overline{\TMF^{\mathcal{B}C_n}}.
\]
In this paper, we are only interested in the case $n = 2$.

\subsection{The descent spectral sequence}\label{subsection:dss}
Our main computational tool is the descent spectral sequence, which we recall in this section.

Let $X$ be any non-connective spectral Deligne--Mumford stack and $\mathcal{F}$ a quasi-coherent sheaf on $X$. Let $U \to X$ be an \'etale cover of $X$. Then the sheaf condition tells us
\[
  \Gamma(X; \mathcal{F}) = \operatorname{Tot}(\Gamma(U \times_X \cdots \times_X U; \pi^* \mathcal{F})),
\]
where $U \times_X \cdots \times_X U$ is the \v{C}ech nerve of the cover, and $\pi\colon U \times_X \cdots \times_X U \to X$ is the projection map. The descent spectral sequence is the Bousfield--Kan spectral sequence for the totalization, and the $E_2$ page is given by the \v{C}ech cohomology
\[
  E_2^{s, t} = \check{H}^s(X_{\mathrm{cl}}; \pi_t \mathcal{F})
\]
of the underlying classical stack $X_{\mathrm{cl}}$ with respect to the cover $U$.

For us, we have $X = \M$, and $U = \M_1(3)$, the spectral enhancement of moduli stack of elliptic curves with a $\Gamma_1(3)$-structure (i.e.\ a choice of $3$-torsion point). By \cite[Theorem 7.2]{affineness-chromatic}, the map
\[
  \Gamma\colon \QCoh(\M) \to \Mod_\TMF
\]
is an equivalence of symmetric monoidal categories. Let $i\colon \M_1(3) \to \M$ be the covering map. Then we have a sequence of equivalences
\[
  \QCoh(\M_1(3)) \!\simeq\! \Mod_{i_* \O_{\M_1(3)}}(\QCoh(\M)) \!\simeq\! \Mod_{\TMF_1(3)}(\Mod_\TMF) \!\simeq\! \Mod_{\TMF_1(3)},
\]
where the first equivalence follows from $i$ being affine and \cite[Proposition 2.5.6.1]{sag}\footnote{The statement of \cite[Proposition 2.5.6.1]{sag} refers to spectral Deligne--Mumford stacks, but the proof applies to non-connective ones as well.}. Under this equivalence, the pullback functor $\Mod_{\TMF} \to \Mod_{\TMF_1(3)}$ is given by $\TMF_1(3) \otimes_{\TMF}(-)$.

More generally, we find that
\[
  \QCoh(\M_1(3) \times_\M \cdots \times_\M \M_1(3)) \simeq \Mod_{\TMF_1(3) \otimes_{\TMF} \cdots \otimes_{\TMF} \TMF_1(3)}.
\]
Thus, we have
\[
  \Gamma(\M_1(3) \times_\M \cdots \times_\M \M_1(3), \pi^* \mathcal{F}) \simeq \TMF_1(3) \otimes_{\TMF} \cdots \otimes_{\TMF} \Gamma(\M; \mathcal{F}).
\]
So the descent spectral sequence is also the $\TMF_1(3)$-based Adams spectral sequence in $\Mod_\TMF$.

There is a well-known identification
\begin{lemma}
  The $\TMF_1(3)$-based Adams spectral sequence in $\Mod_\TMF$ is the same as the $BP$-based Adams--Novikov spectral sequence in spectra.
\end{lemma}
We only use this result to apply the machinery of synthetic spectra to the descent spectral sequence; the morally correct approach would be to reproduce the theory of synthetic spectra inside $\Mod_\TMF$, but we'd rather not take that up.

\begin{proof}
  Following \cite[Section 1]{relations-ass}, it suffices to show that any $\TMF_1(3)$-resolution of a $\TMF$-module in $\Mod_\TMF$ is also an $BP$-resolution in $\Sp$. To do so, we have to show that every $\TMF_1(3)$-injective module in $\Mod_{\TMF}$ is $BP$-injective in $\Sp$, and every $\TMF_1(3)$-exact sequence in $\Mod_\TMF$ is $BP$-exact in $\Sp$.

  \begin{enumerate}
    \item We have to show that $\TMF_1(3) \otimes_\TMF X$ is $BP$-injective in $\Sp$ for any $X \in \Mod_\TMF$. Since $\TMF_1(3)$ is complex orientable, there is a homotopy ring map $MU \to \TMF_1(3)$. Thus, $\TMF_1(3) \otimes_\TMF X$ is a homotopy $MU$-module, hence $MU$-injective, hence $BP$-injective.
    \item Since $F(\TMF, -)$ is right-adjoint to the forgetful functor $\Mod_\TMF \to \Sp$, by definition of exactness, it suffices to show that if $X$ is $BP$-injective, then $F(\TMF, X)$ is $\TMF_1(3)$-injective. Again we may assume $X = BP \otimes Y$.

      By \cite[Theorem 1.2]{homology-tmf}, there exists an even spectrum $Z$ such that $\TMF_1(3) = \TMF \otimes Z$. By evenness, $BP$ is a retract of $BP \otimes DZ$. Thus, $F(\TMF, BP \otimes Y)$ is a retract of $F(\TMF, BP \otimes DZ \otimes Y) = F(\TMF \otimes Z, BP \otimes Y) = F(\TMF_1(3), BP \otimes Y)$, which is a $\TMF_1(3)$-module, hence $\TMF_1(3)$-injective.
  \end{enumerate}
\end{proof}

For convenience, set
\[
  A = \pi_* \TMF_1(3) ,\quad \Gamma = \pi_* \TMF_1(3) \otimes_\TMF \TMF_1(3).
\]
Then $(\Gamma, A)$ is a Hopf algebroid, and for any $\TMF$-module $N = \Gamma(\M; \mathcal{F})$, we have
\[
  \Ext^s_\Gamma(A, \pi_t (\TMF_1(3) \otimes_\TMF N)) = \Ext^s_\Gamma(A, \pi_t (i^* \mathcal{F})) \Rightarrow \pi_{t - s} N.
\]

To perform calculations, it is of course necessary to identify $(\Gamma, A)$ explicitly. From \cite{tmf-level-3}, we have
\[
  A \equiv \pi_* \TMF_1(3) = \Z_2[a_1, a_3, \Delta^{-1}],\quad \Delta = a_3^3 (a_1^3 - 27 a_3),\quad |a_i| = 2i,
\]
with associated elliptic curve
\[
  \E': y^2 z + a_1 xyz + a_3 yz^2 = x^3.
\]

$\Spec \Gamma$ is the classifying scheme of two curves of the form $\E'$ that are abstractly isomorphic, i.e.\ related by a coordinate transform. Consider the change of coordinates
\[
  \begin{aligned}
    x &\mapsto x + rz\\
    y &\mapsto y + sx + tz
  \end{aligned}
\]
In order to preserve the form of the equation, we need
\[
  \begin{aligned}
    0 &= 3r - s^2 - a_1 s\\
    0 &= s^4 - 6 st + a_1 s^3 - 3 a_1 t - 3 a_3 s\\
    0 &= s^6 - 27 t^2 + 3 a_1 s^5 - 9 a_1 s^2 t + 3 a_1^2 s^4 - 9 a_1^2 st + a_1^3 s^3 - 27 a_3 t
  \end{aligned}
\]
So we have $\Gamma = A[s, t]/I$, where $I$ is the ideal generated by the relations above (we have eliminated $r$ entirely). One checks that $\Gamma$ is the free $A$-module on $\{1, s, s^2, s^3, t, st, s^2 t, s^3 t\}$, and these generators exhibits $\TMF_1(3) \otimes_\TMF \TMF_1(3)$ as the sum of $8$ suspended copies of $\TMF_1(3)$.

We can read off the structure maps of the Hopf algebroid to be
\[
  \begin{aligned}
    \eta_R(a_1) &= a_1 + 2s\\
    \eta_R(a_3) &= a_3 + a_1 r + 2t \\
    \Delta(s) &= s \otimes 1 + 1 \otimes s\\
    \Delta(r) &= r \otimes 1 + 1 \otimes r\\
    \Delta(t) &= t \otimes 1 + 1 \otimes t + s \otimes r.
  \end{aligned}
\]

This Hopf algebroid (or rather, the connective version without inverting $\Delta$) was studied in detail in \cite{tilman-tmf}, whose computations and names we will use significantly.

\section{The \texorpdfstring{$E_2$}{E2} page of the DSS}\label{section:e2}
\subsection{Computing the comodule}
Let $q\colon \E' \to \M_1(3)$ be the canonical elliptic curve over $\M_1(3)$, so that we have a pullback diagram
\[
  \begin{tikzcd}
    \E' \ar[d, "q"] \ar[r, "j"] & \E \ar[d, "p"]\\
    \M_1(3) \ar[r, "i"] & \M.
  \end{tikzcd}
\]
Then we have
\[
  \TMF_1(3) \otimes_\TMF \TMF^{\mathcal{B}C_2} = \Gamma(i^* p_* \mathcal{O}_{\E[2]}) = \Gamma(q_* \mathcal{O}_{\E'[2]}) = \TMF_1(3)^{\mathcal{B}C_2},
\]
and similarly with the bar version.

In this section, we compute $\pi_* \TMF_1(3)^{\mathcal{B}C_2}$ as a $\Gamma$-comodule, and then quotient out the image of $1$ and $\tr$. By \Cref{cor:underlying-torsion}, $\pi_* \TMF_1(3)^{\mathcal{B}C_2}$ is given by (the global sections of) the classical scheme of $2$-torsion points of $\E'$.

The na\"ive way to compute $\E'[2]$ is to write down the duplication formula for $\E'$ and compute its kernel. However, the duplication formula is unwieldy. Instead, we write down the inversion map $i\colon \E' \to \E'$ and compute the equalizer with the identity map. The inversion map is induced by the map of projective spaces
\[
  \begin{aligned}
    \P^2 &\to \P^2\\
    [x:y:z] &\mapsto [x:-y - a_1 x - a_3 z: z]
  \end{aligned}
\]
(we use $z$ instead of $z$ since we shall soon use $z$ to mean $-\frac{x}{y}$). The equalizer of $i$ with the identity is then cut out by the equations
\[
  \begin{aligned}
    x(2y + a_1 x + a_3 z) &= 0\\
    z(2y + a_1 x + a_3 z) &= 0
  \end{aligned}
\]

Now observe that the $2$-torsion points are contained in the affine chart $y = 1$. Indeed, if $y = 0$, then the equation defining $E$ tells us $x = 0$. So the unique point on the curve when $y = 0$ is $[0:0:1]$. But this doesn't satisfy the last equation above since $a_3$ is invertible.

Therefore, we work in the $y = 1$ chart. Following standard conventions, we redefine
\[
  z = -\frac{x}{y},\quad w = -\frac{z'}{y},
\]
where $z'$ is the old $z$.

In the new coordinate system, the $2$-torsion points are cut out by the equations
\[
  \begin{aligned}
    z^3 - w + a_1 zw + a_3 w^2 &= 0\\
    2z - a_1 z^2 - a_3 zw &= 0\\
    2w - a_1 zw - a_3 w^2 &= 0.
  \end{aligned}
\]
Adding the first and last equation gives
\[
  z^3 + w = 0.
\]
Eliminating $w$, we find that
\[
  \E'[2] = \Spec A[z]/(2z - a_1 z^2 + a_3 z^4).
\]
In other words,
\[
  \pi_* \TMF_1(3)^{\mathcal{B}C_2} = A[z]/(2z - a_1 z^2 + a_3 z^4),\quad |z| = -2.
\]
Since $a_3$ is invertible, this is a free $A$-module of rank $4$.

The $\Gamma$-coaction on $\pi_* \TMF_1(3)^{\mathcal{B}C_2}$ comes directly from the construction of $\Gamma$ itself; it is given by
\[
  z = -\frac{x}{y} \mapsto -\frac{x + rz'}{y + sx + tz'} = \frac{-\frac{x}{y} - r \frac{z'}{y}}{1 + s \cdot \frac{x}{y} + t \cdot \frac{\tilde{z}}{y}} = \frac{z + rw}{1 - sz - tw} =\frac{z - r z^3}{1 - sz + tz^3}.
\]

\begin{theorem}\label{thm:tr-value}
  The map $\tr\colon \TMF_1(3) \to \TMF_1(3)^{\mathcal{B}C_2}$ sends $1$ to $2 - a_1 z + a_3 z^3$. In particular, by naturality, $2 - a_1 z + a_3 z^3$ is a permanent cocycle.
\end{theorem}

The argument is similar to \cite[Satz 4]{dieck72} (see also \cite[Remark 6.15]{hkr}).
\begin{proof}
  This map is a map of $\TMF_1(3)^{\mathcal{B}C_2}$-modules. Since $z$ acts trivially on $\pi_* \TMF_1(3)$, this means $z \tr 1 = 0$. Since $A[z]$ is a UFD, we know that $\tr 1$ must be a multiple of $2 - a_1 z + a_3 z^3$. Moreover, since it is equal to $2$ after modding out by $z$, the multiple must be $1$.
\end{proof}

So after taking the cofiber by $1$ and $\tr$, we get $\pi_* \overline{\TMF_1(3)^{\mathcal{B}C_2}} = A\{z, z^2\}$, and the $E_2$ page of the descent spectral sequence for $\pi_* \overline{\TMF^{\mathcal{B}C_2}}$ is given by $\Ext_\Gamma(A, A\{z, z^2\})$.

\subsection{Computing the cohomology mod \texorpdfstring{$2$}{2}}
While the coaction itself is fairly complicated, there is a major simplification after we reduce mod $2$.

By computer calculation (\Cref{section:sage}), we find that:
\begin{lemma}\label[lemma]{lemma:computer-action}
  Let
  \[
    b_1 = a_3 z^2,\quad b_5 = a_3^2 z; \quad |b_i| = 2i.
  \]
  Then $A\{z, z^2\} = A\{b_1, b_5\}$, and there is a short exact sequence of comodules
  \[
    0 \to A\{b_1\}/2 \to A \{b_1, b_5\}/2 \to A\{b_5\}/2 \to 0,
  \]
  where both ends are cofree on the indicated generator, inducing a long exact sequence in $\Ext$.

  More precisely, the class $b_1 \in A\{z, z^2\} / 2$ is invariant, while
  \[
    \psi(b_5) - [1] b_5 = [r^2] b_1.
  \]
  Thus, the connecting map of the long exact sequence in $\Ext$ is given by $b_5 \mapsto [r^2] b_1$.
\end{lemma}

For reference, we display the cohomology of $A / 2$ in \Cref{fig:cohomology-a-mod-2}, as computed by \cite{tilman-tmf}. This chart is read as follows:
\begin{itemize}
  \item Each dot represents a copy of $\F_2$; $h_1$-multiplication and $h_3$-multiplication are denoted by lines of slope $1$ and $1/3$ respectively.
  \item $[r^2]$ represents the class $x$ in bidegree $(7, 1)$. This class is uniquely characterized by the fact that $a_1^2$ kills this class, coming from the cobar differential
\[
  \d(a_3^2) = [a_1^2 r^2].
\]
  \item The long dotted lines denote the extension $h_1^4 = a_1^4 \Delta^{-1} g$.
  \item The classes fading out continue $a_1$-periodically, and each ``period'' consists of an infinite $h_1$ tower.
\end{itemize}

In \Cref{fig:ext-connecting}, we put two copies of this next to each other and draw the connecting differential. The resulting cohomology is in \Cref{fig:ext-mod-2}. The hidden extensions follow from a ``multiplication by $\sqrt{\Delta}$'' operator, which we shall next explain.

\DeclareSseqCommand\htwo {} {
  \class (\lastx + 3, \lasty + 1) \structline[htwo]
}
\DeclareSseqCommand\hone {} {
  \class (\lastx + 1, \lasty + 1) \structline[hone]
}
\DeclareSseqCommand\hzero {} {
  \class (\lastx, \lasty + 1) \structline[hzero]
}
\DeclareSseqCommand\honedot {m} {
  \savestack
  \foreach \n in {1,...,#1} {
    \hone
  }
  \draw [->](\lastclass0) -- ++(0.5, 0.5);
  \restorestack
}
\DeclareSseqCommand\honei {} {
  \class (\lastx - 1, \lasty - 1) \structline[hone]
}

\DeclareSseqGroup\tmfAnssModTwoUnit {m} {
  \class (0, 0);
  \draw [->] (0, 0) -- (0.8, 0.8);
  \htwo \htwo \htwo
  \honei \honei
  \htwo
  \honei

  \ifnum#1>0
    \foreach \n in {2, 4, 6, 8} {
      \SseqParseInt\ten{\n * 10}
      \class [white!\ten!black](\n, 0);
      \draw [->, white!\ten!black] (\n, 0) -- (\n + 0.8, 0.8);
    }
  \fi

  \class (15, 3)
  \honei
  \htwo
  \honei \honei
  \htwo \htwo
}

\DeclareSseqGroup\tmfAnssModTwoLayer{m} {
  \tmfAnssModTwoUnit(-24, 0){#1}
  \tmfAnssModTwoUnit(0, 0){#1}
  \tmfAnssModTwoUnit(24, 0){#1}
}

\DeclareSseqGroup\tmfAnssModTwo {m} {
  \tmfAnssModTwoLayer{#1}
  \tmfAnssModTwoLayer(-4, 4){#1}

  \ifnum#1>0
    \draw [dashed] (24, 0) -- (28, 4);
    \draw [dashed] (-24, 0) -- (-20, 4);
    \draw [dashed] (0, 0) -- (4, 4);
  \fi
}
\begin{figure}[ht]
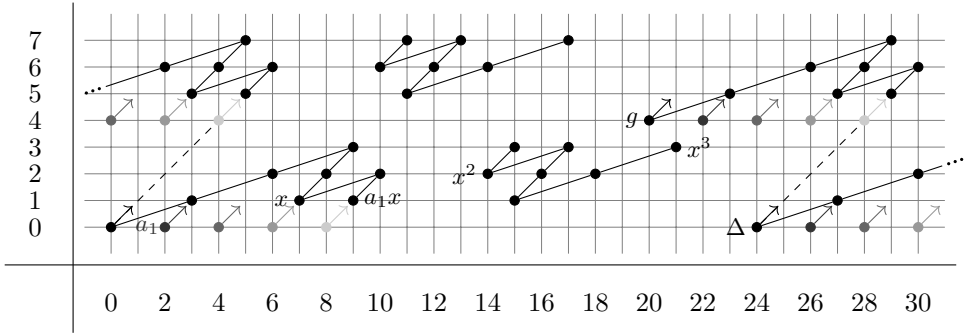

  \centering
  \begin{sseqpage}[small, x range={0}{30}]
    \tmfAnssModTwo{1}

    \classoptions["x"](7, 1);
    \classoptions["a_1 x" {right = 0cm}](9, 1);
    \classoptions["x^2"](14, 2);
    \classoptions["x^3" {right = 0cm}](21, 3);

    \classoptions["a_1" {left=-0.07cm}](2, 0)
    \classoptions["g"](20, 4);
    \classoptions["\Delta"](24, 0);
  \end{sseqpage}
  \caption{Cohomology of $A/2$}\label{fig:cohomology-a-mod-2}
\end{figure}

\begin{figure}[ht]
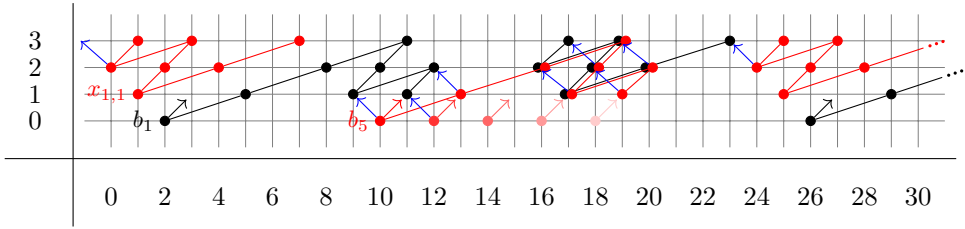

  \centering
  \begin{sseqpage}[small, x range={0}{30}]
    \tmfAnssModTwoLayer(2, 0){0}
    \tmfAnssModTwoLayer[red](10 - 24, 0){0}

    \foreach \n in {2, 4, 6, 8} {
      \SseqParseInt\ten{\n * 10}
      \class [white!\ten!red](\n + 10, 0);
      \draw [->, white!\ten!red] (\n + 10, 0) -- (\n + 10.8, 0.8);
    }

    \class(-1, 3) 
    \d1(0, 2)
    \d1(10, 0)
    \d1(12, 0)
    \d1(13, 1)

    \d1(17, 1, 2, 1)
    \d1(18, 2, 2)
    \d1(19, 1, 1, 1)
    \d1(20, 2, 2, 1)
    \d1(24, 2)

    \classoptions["b_1"](2, 0)
    \classoptions["b_5"](10, 0)

    \classoptions["x_{1, 1}"](1, 1)
  \end{sseqpage}
  \caption{Connecting maps for $\Ext_\Gamma(A, A\{b_1, b_5\} / 2)$}\label{fig:ext-connecting}
\end{figure}

\DeclareSseqGroup\TmfCtwoAnssModTwo {} {
  \foreach \m in {0, 24} {
    \class (2 + \m, 0)
    \draw [->] (2 + \m, 0) -- (2.8 + \m, 0.8);
    \htwo \htwo

    \foreach \n in {2, 4, 6, 8} {
      \SseqParseInt\ten{\n * 10}
      \class [white!\ten!black](\n + \m + 2, 0);
      \draw [->, white!\ten!black] (\n + \m + 2, 0) -- (\n + \m + 2.8, 0.8);
    }
  }

  \class(17, 1) \htwo

  \begin{scope}[red]
    \foreach \m in {0, 24} {
      \class (1 + \m, 3)
      \class (3 + \m, 3)
      \honei \honei
      \htwo \htwo
    }

    \foreach \n in {0, 2, 4, 6, 8} {
      \SseqParseInt\ten{\n * 10}
      \class [white!\ten!red](\n + 14, 0);
      \draw [->, white!\ten!red] (\n + 14, 0) -- (\n + 14.8, 0.8);
    }

    \class (13, 3)
    \class (13, 1)
    \draw [->] (13, 1) -- (13.8, 1.8);
    \htwo \htwo
  \end{scope}
}

\begin{figure}[ht]
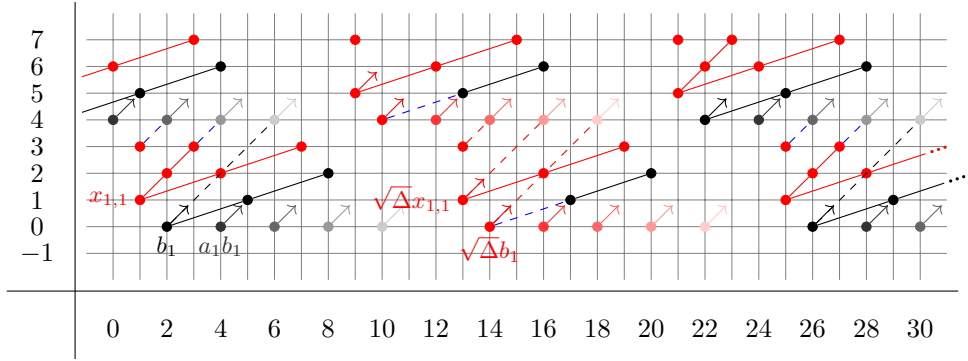

  \centering
  \begin{sseqpage}[small, x range={0}{30}, yrange={-1}{7}]
    \TmfCtwoAnssModTwo
    \TmfCtwoAnssModTwo(-4, 4)

    \draw[dashed, red] (13, 3) -- (14, 4);
    \draw[dashed, red] (13, 1) -- (16, 4);
    \draw[dashed, red] (14, 0) -- (18, 4);

    \draw[dashed](2, 0) -- (6, 4);
    \draw[dashed](26, 0) -- (30, 4);

    \draw[dashed, blue] (3, 3) -- (4, 4);
    \draw[dashed, blue] (1, 3) -- (2, 4);
    \draw[dashed, blue] (24 + 3, 3) -- (24 + 4, 4);
    \draw[dashed, blue] (24 + 1, 3) -- (24 + 2, 4);

    \draw[dashed, blue] (14, 0) -- (17, 1);
    \draw[dashed, blue] (10, 4) -- (13, 5);

    \classoptions["x_{1, 1}"](1, 1)
    \classoptions["b_1" {below = -0.5pt}](2, 0)
    \classoptions["a_1 b_1" {below = -0.5pt}](4, 0)
    \classoptions["\sqrt{\Delta} x_{1, 1}"](13, 1)
    \classoptions["\sqrt{\Delta} b_1" {below = -0.5pt}](14, 0)
  \end{sseqpage}
  \caption{$\Ext_\Gamma(A, A\{b_1, b_5\}/2)$}\label{fig:ext-mod-2}
\end{figure}

Originally, we have an action of $A[z]/(2z - a_1 z^2 + a_3 z^4)$ on $A\{z, z^2\} / 2$. Since we have quotiented out by $2$ and $2 - a_1 z + a_3 z^3$, this reduces to an action by $A [z]/(2, a_1 z + a_3 z^3)$. Further, since we are acting on $z$-multiples only, this reduces to an action by $A[z]/(2, a_1 + a_3 z^2)$. In this ring, we have
\[
  \Delta = a_3^4 + a_3^3 a_1^3 = a_3^4 + a_3^4 a_1^2 z^2 = (a_3^2 (1 + a_1 z))^2.
\]

One can check via \texttt{sage} that $\sqrt{\Delta} = a_3^2 (1 + a_1 z)$ is invariant in $A[z]/(2, z + a_3 z^2)$, so acts on $\Ext_\Gamma(A, A\{z, z^2\}/ 2)$ (see again \Cref{section:sage}). For example, the surviving class in bidegree $(14, 0)$ is $\sqrt{\Delta} b_1 = a_3^3z^2 - a_3^2 a_1^2 z$.

\subsection{\texorpdfstring{$2$}{2}-Bockstein spectral sequence}
We now run the $2$-Bockstein spectral sequence, which we will find to degenerate on the $E_2$ page. These Bockstein $d_1$'s resemble the $d_3$'s in the descent spectral sequence quite a bit. Thus, despite the fact that a lot of the differentials can be computed by writing down explicit cocycles, we try our best to argue them formally so that the same argument can be applied to the $d_3$'s.

Looking at the chart in \Cref{fig:ext-mod-2}, it is not hard to see what to expect. All differentials have bidegree $(-1, 1)$, and we know that nothing above the zero line survives, since $(\Gamma, A)$ has no rational cohomology. Thus, for example, up to $O(a_1)$, the class in bidegree $(1, 1)$ must be hit by a differential from $b_1$. The main work to do is to make sure nothing exotic happens with the highly $a_1$-divisible classes coming from $\Delta$ division.

To begin, recall that in the $2$-Bockstein for $\Ext_\Gamma(A, A)$, we have
\[
  d_1(a_1) = h_1,
\]
since $h_1 = [s]$ and $\eta_R(a_1) = a_1 + 2s$.

\begin{lemma}\label[lemma]{lemma:no-h1}
  There are no non-zero classes of the form $h_1^2 a$ on the $E_2$ page. Further, any permanent class of this form equals $d_1(h_1 a_1 a)$.
\end{lemma}

\begin{proof}
  If $d_1(h_1^2 a) \not= 0$, then it doesn't survive. Otherwise, consider $d_1(a)$. This must be $h_1^2$ torsion, so it is an $h_2$ multiple. Then $h_1 d_1(a) = 0$. So $d_1(h_1 a_1 a) = h_1^2 a$.
\end{proof}

In general, we let $x_{t - s, s}$ denote a class in the corresponding bidegree that generates the bidegree after modding out by $a_1$- and $h_1$-multiples, if this makes sense. This class is well-defined up to $a_1$- and $h_1$-multiples.

\begin{lemma}
  $d_1(b_1) = x_{1, 1}$ and $d_1(\sqrt{\Delta} b_1) = \sqrt{\Delta} x_{1, 1}$.
\end{lemma}
Note that since $x_{1, 1}$ is only well-defined up to $a_1$ multiples, this is equivalent to saying that $d_1(b_1)$ and $d_1(\sqrt{\Delta}) b_1$ are not $a_1$ divisible. Since $x_{1, 1}$ is not well-defined, neither is $\sqrt{\Delta} x_{1, 1}$, and we are not claiming that there is a single choice of $x_{1, 1}$ for which both equations hold.

\begin{proof}
  First observe that there is a choice of $x_{1, 1}$ that is permanent. Indeed, for any choice of $x_{1, 1}$, the class $d_1(x_{1, 1})$ must be $h_1^2$-divisible, so it must be hit by a $d_1$ from an $a_1$-multiple by \Cref{lemma:no-h1}, which we can add to $x_{1, 1}$, so that it survives the $d_1$. From the $E_2$ page onwards, the target bidegree of the differential is $0$ by \Cref{lemma:no-h1} again.

  Now $h_1^2 x_{1, 1}$ must be hit by a $d_1$, and the source can only be $h_1^2 b_1 + O(a_1)$, since $h_2 x_{1, 1}$ is permanent and other classes are highly $a_1$-divisible. So $d_1(b_1)$ must hit a version of $x_{1, 1}$.

  The case of $\sqrt{\Delta} b_1$ is analogous.
\end{proof}

\begin{corollary}
  There is a hidden extension $a_1 x_{1, 1} = h_1 b_1 + O(a_1^2)$.
\end{corollary}

\begin{proof}
  The bidegree $(3, 1)$ is generated by $h_1 b_1$ and $a_1^2$ multiples. So $a_1 x_{1, 1}$ is either $h_1 b_1 + O(a_1^2)$ or $O(a_1^2)$. But $d_1(a_1 x_{1, 1}) = h_1 x_{1, 1}$ is not an $a_1^2$-multiple, so $a_1 x_{1, 1}$ cannot be $O(a_1^2)$. So it must be $h_1 b_1 + O(a_1^2)$.
\end{proof}

\begin{lemma}\label[lemma]{lemma:juggle}
  $d_1(a_1 b_1) = d_1(a_1 \sqrt{\Delta} b_1) = 0$.
\end{lemma}

\begin{proof}
  Note that $b_1$ and $\sqrt{\Delta} b_1$ generate the $0$-line under $a_1$ and $\Delta^{\pm}$, and $d_1(a_1 b_1)$ and $d_1(a_1 \sqrt{\Delta} b_1)$ must be in the submodule generated by these and $h_1$. We first show that the values of the differentials must be $a_1^2$-divisible. Indeed, we cannot have $d_1(a_1 b_1) = h_1 b_1 + O(a_1^2)$, because applying $d_1$ again would imply that
  \[
    0 = h_1 x_{1, 1} + O(a_1^2),
  \]
  a contradiction. The argument for $d_1(a_1 \sqrt{\Delta} b_1)$ is similar.

  Set $\Delta = 1$, and let $x = a_1 b_1$, $y = a_1 \sqrt{\Delta} b_1$. Then we can write
  \[
    d_1 \begin{pmatrix}x\\y \end{pmatrix} = h_1 M \begin{pmatrix}x\\y \end{pmatrix}
  \]
  for some matrix $M$ of odd polynomials in $a_1$. Applying $d_1$ again gives us the equation
  \[
    M^2 = \frac{1}{a_1} M.
  \]
  In other words, we have $M = a_1 M^2$. Iterating this equation shows that $M$ is infinitely $a_1$-divisible, so it must be trivial.
\end{proof}

\begin{corollary}\label[corollary]{cor:cross}
  $a_1 d_1(b_1) = h_1 b_1$, or equivalently, $a_1 x_{1, 1} = h_1 b_1 + O(a_1^2)$.
\end{corollary}

\begin{remark}
  We can in fact write down explicit lifts of $a_1 b_1$ and $a_1 \sqrt{\Delta} b_1$, namely
  \[
    a_1 b_1 + 2 a_3 z,\quad a_1 \sqrt{\Delta} b_1 + 2a_3^3 z,
  \]
  whose coboundary vanishes mod $2^2$. However, the proofs above will be used for $d_3$'s in the descent spectral sequence too, and we cannot write down explicit cocycles for that.
\end{remark}

With $\Delta^{\pm}$ and $g$ periodicity, this gives all $d_1$'s. No classes are left in positive $s$ so we are done. The resulting $E_2$ page of the descent spectral sequence of $\overline{\TMF^{\mathcal{B}C_2}}$ has a fairly regular pattern, which we exhibit in \Cref{fig:anss-e2-page}. The names are intentionally left off; they can be found in \Cref{fig:anss-e3-page}.

\begin{sseqdata}[name = anss e3, small, x range={0}{30}, y range={0}{7}]
  \foreach \x in {0, 12, 24} {
    \class(\x + 1, 1) \honedot{5};
    \htwo \htwo
    \class(\x - 3, 5) \honedot{2};
    \htwo \htwo

    \class[rectangle, fill=white] (\x + 4, 0);
    \honedot{3}
    \begin{scope}[white!40!black]
      \class[rectangle, fill=white] (\x + 8, 0);
      \honedot{3}
    \end{scope}
    \begin{scope}[white!80!black]
      \class[rectangle, fill=white] (\x + 12, 0);
      \honedot{3}
    \end{scope}

    \class (\x, 4) \honedot{2}
  }
  \d3(4, 0)
  \d3(5, 1)
  \d3(6, 2)
  \d3(7, 3, 2)

  \d[white!80!blue]3(12, 0)

  \d3(13, 1, 2)
  \d3(14, 2, 2)
  \d3(15, 3, 2)

  \d3(4 + 24, 0)(27, 3, 2)
  \d3(5 + 24, 1)
  \d3(6 + 24, 2)

  \d3(24, 4)
  \d3(20, 0)(19, 3, 2)

  \structline[dashed, bend left=30, page=3, source anchor=70, target anchor=200] (19, 3, 1)(22, 6)
\end{sseqdata}

\begin{figure}[H]
  \centering
  \printpage[name = anss e3, page = 2]
  \caption{$E_2$ page of ANSS}\label{fig:anss-e2-page}
\end{figure}

\section{Differentials in the DSS}\label{section:differentials}
We have now computed the $E_2$ page of the descent spectral sequence of $\overline{\TMF^{\mathcal{B}C_2}}$. The goal of this section is to compute the differentials.

The main difficulty in computing the descent spectral sequence differentials is translational invariance --- the $E_2$ page is $\Delta$-invariant, but the $E_\infty$ page will only be $\Delta^8$-invariant. If we had a connective version, then the leftmost class must be permanent since there is nothing to hit. Since we do not, we need external means of determining that certain classes are permanent. Once we do so, we can use standard techniques in homotopy theory to compute the remaining differentials.

We begin by computing the $d_3$'s, where most of the hard work lies in. We depict the end result in \Cref{fig:anss-e3-page} for reference.

\begin{figure}[ht]
  \centering
  \begin{sseqpage}[name = anss e3, page = 3]
    \classoptions["h_1^2 t"](3, 3)
    \classoptions["\sqrt{\Delta} t"](13, 1, 2)
    \classoptions["\Delta t"](25, 1, 2)

    \classoptions["x_{12, 4}"](12, 4)

    \classoptions["x_{4, 0}"](4, 0)
    \classoptions["x_{8, 0}"](8, 0)
    \classoptions["x_{16, 0}"](16, 0)

    \classoptions["gt" {below = -0.5pt}](21, 5)
    \classoptions["gx_{4, 0}" {right = 0pt}](24, 4)
    \classoptions["gx_{8, 0}" {right = 0pt}](28, 4)

    \sseqset{class label handler = { \def\result{\scalebox{0.5}{$#1$}} }}
    \classoptions["h_2 \sqrt{\Delta} t" {below = -0.5pt}](16, 2)
    \sseqset{class label handler = { \def\result{\scalebox{0.6}{$#1$}} }}
    \classoptions["{x_{1, 1}\!=\!t\!}"](1, 1)
  \end{sseqpage}
  \caption{$E_3$ page of the descent spectral sequence}\label{fig:anss-e3-page}
\end{figure}

To compute the $d_3$'s, we have to show that $x_{1, 1}$ is permanent by explicitly constructing a homotopy class $t$, while $\sqrt{\Delta} x_{1, 1}$ supports a $d_3$. The rest then follows formally using $\eta^4 = 0$. Along the way, we will find a hidden $\nu$-extension from $h_2^2 \sqrt{\Delta} t$ to $h_1 gt$, which will be useful later on.

For the purposes of computing $d_3$, it is convenient to have a ``multiplication by $v_1^2$'' operation, obtained by lifting $a_1^2$-multiplication on the mod $2$ cohomology. This is well-defined up to $O(2)$, which is fine because the targets of all differentials are $2$-torsion. Then $c_4 = v_1^4$ and $c_6 = v_1^6$, which lets us deduce that $d_3(v_1^2) = h_1^3$.
\begin{theorem}\label{thm:x11-permanent}
  There is a choice of $x_{1, 1}$ that survives and has order $2$. We call this class $t$.
\end{theorem}
In fact, all choices survive and have order $2$, but we will only get to see this after computing the spectral sequence fully.

\begin{proof}
  We define $t$ to be the composition
  \[
    t\colon \Sigma \TMF \hookrightarrow \TMF \otimes \RP^\infty_+ = \TMF_{h C_2} \overset{\mathrm{Nm}}\longrightarrow \TMF^{\mathcal{B}C_2} \to \TMF^{\mathcal{B}C_2} / (1, \tr),
  \]
  where the first map is the inclusion of the $1$-cell. We claim this this has ANSS filtration $1$ and is non-$v_1^4$-divisible on the $E_2$ page (note that there are no Adams filtration 0 elements in odd degrees). Then it must be detected by a choice of $x_{1, 1}$.

  To do so, consider the composite
  \begin{multline*}
    t'\colon \Sigma \TMF \overset{t}\to \TMF^{\mathcal{B}C_2}/(1, \tr) \\
    \to \TMF^{h C_2}/(1, \tr) = \TMF^{\RP^\infty_+}/(1, \tr) \to \TMF^{\RP^2_+}/(1, \tr),
  \end{multline*}
  where we use $\tr$ to refer to the composite $\Sigma \TMF \to \TMF^{\RP^\infty_+} \to \TMF^{\RP^2_+}$ as well. It suffices to prove the same properties for $t'$.

  The key fact from equivariant homotopy theory we use is the following: let $X$ be a genuine $C_2$-spectrum whose $C_2$ action on the underlying spectrum $\iota X$ is trivial. Then the cofiber of the composition
  \[
    \iota X \otimes \RP^n_+ \to \iota X \otimes \RP^\infty_+ = X_{h C_2} \overset{\mathrm{Nm}}\to X^{h C_2} = (\iota X)^{\RP^\infty_+} \to (\iota X)^{\RP^m_+}
  \]
  is $\Sigma \iota X \otimes P_{-m - 1}^n$, where $P_{-m}^n$ is the stunted projective space. This is well-known, but we are unable to find a reference, so we prove this in \Cref{section:stunted}.

  Take $n = 1$, so that $\RP^n_+ = S^0 \vee S^1$. The cell diagram of  $\Sigma P_{-3}^1$ is given by
  
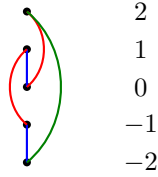
\begin{figure}[H]
    \centering
    \begin{tikzpicture}[scale=0.5]
      \begin{scope}[xshift=0cm]
        \begin{celldiagram}
          \foreach \k in {-2,...,2} {
            \n\k
            \node at (3, \k) {$\k$};
          }

          \two{0} \two{-2} \eta{0} \eta{-1} \nu{-2}
        \end{celldiagram}%
      \end{scope}%
    \end{tikzpicture}%
  \end{figure}%
  \noindent where as usual the attaching maps of degree $1, 2, 4$ are $2, \eta, \nu$ respectively. In this diagram, we think of each cell as a $\TMF$-module cell, i.e.\ a copy of $\TMF$.

  We can read off all the information we need from this diagram. We start with $\TMF^{\RP^2_+}$, which is the bottom three cells in the diagram. We first understand what happens when we mod out $1$ and $\tr$.

  Recall that $1$ is the global sections of the projection map $\E[2] \to \M$, which is split by the identity section $\M \to \E[2]$. The global sections of the identity section is the inclusion of the fixed points $\TMF^{\mathcal{B}C_2} \to \TMF^{h C_2} \to \TMF$ by construction of equivariant TMF. That is, $1$ is a section of the projection $\TMF^{\RP^2_+} \to \TMF$ onto the $0$-cell. Thus, quotienting out by $1$ kills off the $0$-cell, and we are left with the bottom two cells.

  By construction, $\tr$ is the attaching map of the $1$-cell. Thus, further quotienting by $\tr$ adds the $1$-cell, and $\TMF^{\RP^2_+}/(1, \tr)$ is the question mark complex, i.e.\ the subcomplex consisting of the $(-2)$-, $(-1)$- and $1$-cell.

  Finally, $t'$ is the attaching map of the $2$-cell. It must factor through the bottom cell since $\pi_2 \TMF / \eta = 0$, and the diagram tells us this map is $\nu$ on the bottom cell, as desired.
\end{proof}

\begin{corollary}\label[corollary]{cor:d3x40}
  There is a choice of $x_{4, 0}$ with $d_3(x_{4, 0}) = h_1^2 t$.
\end{corollary}

\begin{proof}
  Since $\eta^4 = 0$, we know that $h_1^4 t$ must be hit by a differential. The only possible source is $h_1^2 x_{4, 0}$. So we are done. 
\end{proof}

Our next goal is to show that $\sqrt{\Delta} t$ is \emph{not} permanent, and instead supports a $d_3$. We can think of this as a $d_3$ on the hypothetical $\sqrt{\Delta}$ (which, if existed, must support a $d_3$ since $\sqrt{\Delta}^2$ supports a non-$2$-divisible $d_5$). The proof is somewhat roundabout.

Since $t$ is $2$-torsion, we get a map $\Sigma \TMF/2 \to \overline{\TMF^{\mathcal{B}C_2}}$ picking out $t$. The homotopy groups of $\tmf / 2$ up to the $20$\textsuperscript{th} stem are depicted in \Cref{fig:anss-tmf-mod-2}. We name these classes as follows --- if $y \in \pi_* \TMF$ is $2$-torsion, we let $\tilde{y} \in \pi_* \TMF/2$ be the class that is $y$ on the top cell. This is well-defined up to an element in the image of $\pi_* \TMF$. In particular, we are interested in the following classes:
\begin{enumerate}
  \item $\kappa \in \pi_{14} \TMF$ is well-defined, while $\kappabar \in \pi_{20} \TMF$ is well-defined mod $2$.
  \item $\tilde{\nu^2} \in \pi_7 \TMF /2 \cong \Z/2$ is the unique non-zero element in this degree.
  \item $\tilde{\kappa} \in \pi_{15} \TMF/2$ is well-defined up to $\eta \kappa$. Thus, $\nu \tilde{\kappa}$ is well-defined.
\end{enumerate}

\DeclareSseqGroup\zigzag {} {
  \class(0, 0)
  \savestack
  \hone \hone
  \class(2, 0) \structline[dashed] \lastclass1 \hone \hone
  \restorestack
}
\begin{figure}[ht]
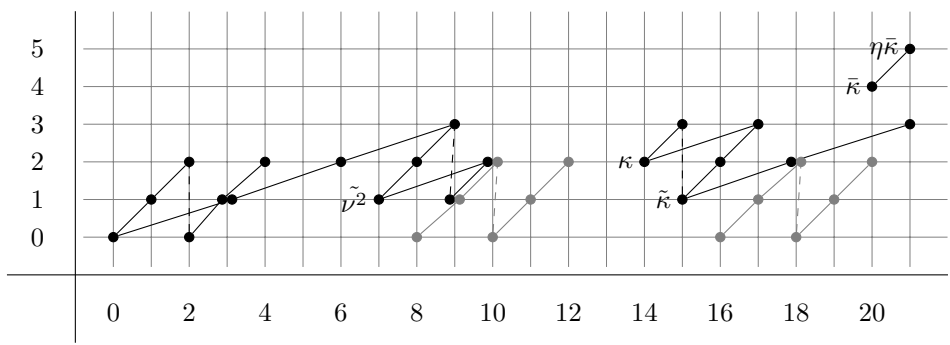

  \centering
  \def\labelscalefactor{1}
  \begin{sseqpage}[small, scale = 1.41, x range={0}{21}]
    \zigzag(0, 0)
    \htwo \htwo \htwo
    \honei \honei
    \htwo \honei \structline[dashed] (9, 3)

    \zigzag[gray](8, 0)

    \class(15, 3) \honei \htwo
    \honei \honei \structline[dashed] (15, 3) \htwo \htwo
    \class(21, 5) \honei

    \zigzag[gray](16, 0)

    \classoptions["\tilde{\nu^2}"](7, 1)
    \classoptions["\kappa"](14, 2)
    \classoptions["\tilde\kappa"](15, 1)
    \classoptions["\kappabar"](20, 4)
    \classoptions["\eta\kappabar"](21, 5)
  \end{sseqpage}
  \caption{$E_\infty$ page of the ANSS of $\TMF/2$}\label{fig:anss-tmf-mod-2}
\end{figure}

\begin{lemma}
  In $\pi_* \TMF/2$, we have
  \[
    \eta \kappabar = \nu^2 \tilde{\kappa} + \kappa \tilde{\nu^2}.
  \]
\end{lemma}

\begin{proof}
  We start with the Adams spectral sequence for $\pi_* \tmf$, which is depicted in \Cref{fig:tmf-ass-e2}. This may be computed by the May spectral sequence or a computer. The only possible $d_2$'s in this range are the ones we have drawn, and any of the differentials implies all others by the Leibniz rule. We can get these via the fact that $v_1^4\nu = 0$, for example.

  From this, Moss' convergence theorem \cite{moss} tells us
  \[
    \eta\kappabar = \langle \kappa, 2, \nu^2\rangle \in \pi_* \TMF
  \]
  with no indeterminacy. By definition, the right-hand side is given by the composite\footnote{In this diagram, the spectrum in the middle is (a shift of) $\TMF/2$ and the ones at the end are $\TMF$. The map on the left is any map such that if you project onto the top cell of $\TMF / 2$, then the map is $\kappa$ (``$\kappa$ on the top cell''); the map on the right is any map such that the restriction to the bottom cell of $\TMF / 2$ is $\nu^2$ (``$\nu^2$ on the bottom cell''). The composite of any two such choices given an element in the Toda bracket, and vice versa.}

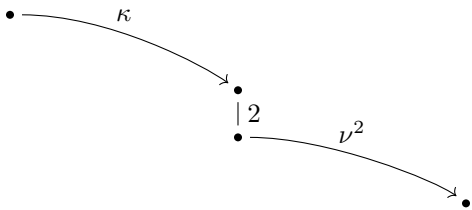
  \begin{figure}[H]
    \centering
    \begin{tikzpicture}[xscale = 3, yscale = 0.5]
      \node [circ, outer sep=3] (a) at (0, 5) {};
      \node [circ, outer sep=3] (b) at (1, 3) {};
      \node [circ, outer sep=3] (c) at (1, 1.75) {};
      \node [circ, outer sep=3] (d) at (2, 0) {};

      \draw (c) -- (b) node [pos=0.5, right] {$2$};

      \draw [->] (a) to [out=0, in=110, looseness=0.5] (b);
      \draw [->] (c) to [out=0, in=110, looseness=0.5] (d);

      \node at (0.5, 5) {$\kappa$};
      \node at (1.5, 1.9) {$\nu^2$};
    \end{tikzpicture}
  \end{figure}

  Let $\overline{\nu^2}\colon \Sigma^6 \TMF / 2 \to \TMF / 2$ be the map that first projects to the top $\TMF$-cell, and then maps via $\tilde{\nu^2}$. Then $\nu^2 - \overline{\nu^2}$ maps trivially to the top cell, so factors through the bottom cell. This factorization is a valid choice of ``$\nu^2$ on the bottom cell''. So
  \[
    \eta \kappabar = (\nu^2 - \overline{\nu^2}) \tilde{\kappa} = \nu^2 \tilde{\kappa} - \overline{\nu^2} \tilde{\kappa}.
  \]
  Finally, note that $\overline{\nu^2} \tilde{\kappa} = \tilde{\nu^2} \kappa$, and that everything is $2$-torsion, so we can drop the signs.
\end{proof}

\begin{figure}[h]
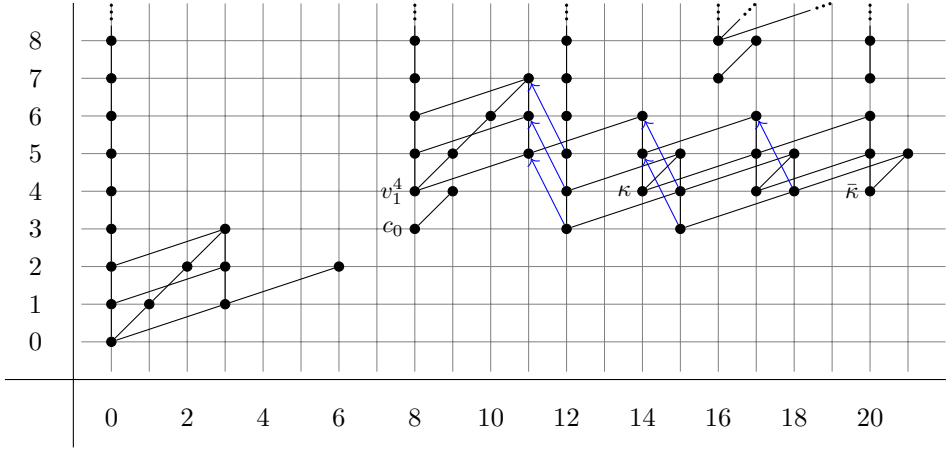

  \begin{sseqpage}[small, x range={0}{21}, y range={0}{8}, scale=1.41]
    \foreach \n in {0,1} {
      \hzerotower(\n * 8, \n * 4)
      \savestack
      \hone \hone \hone 
      \restorestack
      \htwo \htwo 
      \class(3 + \n * 8, 2 + \n * 4)
      \structline(3 + \n * 8, 1 + \n * 4) \structline (3 + \n * 8, 3 + \n * 4) \structline(0 + \n * 8, 1 + \n * 4)
      \structline(3 + \n * 8, 3 + \n * 4)(0 + \n * 8, 2 + \n * 4)

      \class(8 + \n * 8, 3 + \n * 4) \hone

      \hzerotower(12 + \n * 8, 3 + \n * 4)
    }
    \class(16, 8) \hzero 
    \class(17, 9) \structline(16, 8)
    \class(19, 9) \structline(16, 8)

    \classoptions["c_0"](8, 3)
    \classoptions["v_1^4"](8, 4)

    \class(15, 3) \hzero \hzero \honei \hzero

    \class(18, 4) \hzero \honei \hzero \hzero

    \class(21, 5) \honei \hzero \hzero

    \structline(12, 4)(15, 5)
    \structline(12, 3)(15, 4)

    \structline(14, 5)(14, 6)
    \structline(14, 4)(17, 5)
    \structline(14, 5)(17, 6)
    \structline(17, 4)(20, 5)
    \structline(17, 5)(20, 6)

    \structline(15, 3)(18, 4)
    \structline(15, 4)(18, 5)
    \structline(18, 4)(21, 5)

    \classoptions["\kappabar"](20, 4)
    \classoptions["\kappa"](14, 4)

    \d2(12, 3)
    \d2(12, 4)
    \d2(12, 5)

    \d2(15, 3)
    \d2(15, 4)
    \d2(18, 4)
  \end{sseqpage}
  \caption{$E_2$ page of Adams spectral sequence for $\tmf$}\label{fig:tmf-ass-e2}
\end{figure}

\begin{corollary}\label[corollary]{cor:hidden-ext}
  $h_2 \sqrt{\Delta} t$ represents $\tilde{\kappa} t$, and in particular is permanent. Further, there is a hidden $\nu$ extension from $h_2^2 \sqrt{\Delta} t$ to $h_1 g t$.
\end{corollary}

\begin{proof}
  The class $t$ gives a map $\Sigma \TMF / 2 \to \overline{\TMF^{\mathcal{B}C_2}}$. Thus, the previous lemma gives
  \[
    \nu^2 \tilde{\kappa} t = \eta \kappabar t + \tilde{\nu^2} \kappa t \in \pi_* \overline{\TMF^{\mathcal{B}C_2}}
  \]
  We know that $\tilde{\nu^2} t \in \pi_8 \overline{\TMF^{\mathcal{B}C_2}}$ has very high ANSS filtration (at least 7) because there is nothing in lower degrees. So
  \[
    \nu^2 \tilde{\kappa} t = \eta \kappabar t + \text{higher filtration}.
  \]
  $\eta \kappabar t$ is represented by $h_1 g t$, so $\tilde{\kappa} t$ must be detected by a permanent class with filtration at most $4$, which must be $h_2 \sqrt{\Delta} t$ (this can alternatively follow from calculating the products in $\Ext$, but we spare ourselves the trouble).
\end{proof}

Our next result uses the machinery of synthetic spectra, which ``categorifies'' the Adams spectral sequence. They were originally introduced in \cite{synthetic}, but \cite[Section 9]{manifold-synthetic} gives a more computationally-oriented introduction to the subject. Specifically, the precise relationship between synthetic spectra and the Adams spectral sequence is laid out in \cite[Theorem 9.19]{manifold-synthetic}. Further, \cite[Appendix A.2]{manifold-synthetic} gives some helpful example calculations. Note however that we use a different grading convention. Our grading corresponds directly to the $x$ and $y$ coordinates in the Adams chart. The exact conversion is that $\mathbb{S}^{a, b}$ in their grading is $S^{a, b - a}$ in ours.

If $z$ is a cocycle in the $E_2$ page, we let $[z]$ denote any element of $\pi_*$ that is represented by $z$.
\begin{corollary}\label[corollary]{cor:hidden-nu}
  There is a hidden $\nu$ extension from $\Delta^k h_2^2 \sqrt{\Delta} t$ to $\Delta^k h_1 g t$ for every $k$. That is, if $\Delta^k h_2^2 \sqrt{\Delta} t$ is permanent, then $\nu [\Delta^k h_2^2 \sqrt{\Delta} t]$ is detected by $\Delta^k h_1 g t$.
\end{corollary}
We use $h_2$ to denote multiplication on the $E_2$ page, and $\nu$ to denote multiplication on homotopy groups. We insist on distinguishing these since $[h_2^2 \sqrt{\Delta} t]$ is not $\nu^2$-divisible in the homotopy groups.

\begin{proof}
  We work in $BP$-synthetic spectra, and identify $\overline{\TMF^{\mathcal{B}C_2}}$ with its synthetic analogue $\nu (\overline{\TMF^{\mathcal{B}C_2}})$. In synthetic homotopy groups, the previous result can be written as $\nu^2 \tilde{\kappa} t = \tau^2 \eta \kappabar t$.\footnote{To see this, we know that $\nu^2 \tilde{\kappa} t$ is equal to $\eta \kappabar t$ after inverting $\tau$, and $\tau^2$ is the right number of copies of $\tau$ to put the right-hand side in the right bidegree, since the product jumps by two filtrations. We also have to check that there are no $\tau$-torsion classes. $\tau$-torsion classes are generated by classes that are hit by differentials, and a class hit by a $d_k$ is killed by $\tau^{k - 1}$. Thus, the $\tau$-torsion terms in bidegree $(22, 4)$ are classes hit by differentials from $(23, 2)$ or below, of which there are none.}

  Suppose $\Delta^k h_2^2 \sqrt{\Delta} t$ is permanent. Let $\alpha \in \pi_{19 + 24k, 3} \overline{\TMF^{\mathcal{B}C_2}}$ be a class whose image in $\overline{\TMF^{\mathcal{B}C_2}}/\tau$ is $\Delta^k h_2^2 \sqrt{\Delta} t$. Consider its image in $\overline{\TMF^{\mathcal{B}C_2}}/\tau^3$. Since $\Delta$ survives to the $E_5$-page, we know that $\Delta$ lifts to $\pi_{24, 0} \TMF/\tau^3$ (uniquely, since it is on the $0$-line). Since $\tilde{\kappa} t$ represents $h_2 \sqrt{\Delta} t$, we can write\footnote{To multiply, we need to know that $C\tau^n$ is a ring. To see this, note that the natural $t$-structure of \cite[Proposition 2.16]{synthetic} is compatible with the symmetric monoidal structure by \cite[Proposition 2.29]{synthetic}. Further, the proof of \cite[Proposition 4.29]{synthetic} shows that $C\tau^n$ is the $(n - 1)$-truncation of the unit, so it has a natural $\mathbb{E}_\infty$-ring structure.}
  \[
    \alpha = \Delta^k \nu \tilde{\kappa} t \in \pi_{19 + 24k, 3} \overline{\TMF^{\mathcal{B}C_2}}/\tau^3.
  \]
  So we know that
  \[
    \nu \alpha = \Delta^k \nu^2 \tilde{\kappa} t = \tau^2 \Delta^k \eta \kappabar t \in \pi_{22 + 24k, 4} \overline{\TMF^{\mathcal{B}C_2}}/\tau^3.
  \]
  So in $\pi_{*, *} \overline{\TMF^{\mathcal{B}C_2}}$, we know that $\nu \alpha = \tau^2 [\Delta^k h_1 gt] + O(\tau^3)$.
\end{proof}

\begin{lemma}\label[lemma]{lemma:sqrt-delta-t}
  $\sqrt{\Delta} t$ does not survive to the $E_4$ page.
\end{lemma}

Note that if $\sqrt{\Delta} t$ survived to the $E_\infty$ page and is $2$-torsion, then \Cref{cor:hidden-nu} implies $\eta \kappabar t$ is $\nu^3$-divisible. However, $\nu^3$ is $\eta^2$-divisible in $\pi_* \TMF / 2$, and there is no candidate for the $\eta^2$ division of $\eta \kappabar t$ --- the classes on the $0$-line are $\kappabar$-torsion but $\eta \kappabar t$ is not. The proof runs this argument in synthetic spectra to get the stronger claim that it doesn't survive to $E_4$.
\begin{proof}
  We again work in synthetic spectra.

  If $\sqrt{\Delta} t$ survived to the $E_4$ page, then it lifts to a class in $\pi_{13, 1}(\overline{\TMF^{\mathcal{B}C_2}} / \tau^3)$, which we shall call $\sqrt{\Delta} t$ again. Then
  \[
    \nu^3 \sqrt{\Delta} t = \tau^2 \eta \kappabar t \not= 0 \in \pi_{22, 4} \overline{\TMF^{\mathcal{B} C_2}} / \tau^3.
  \]

  Now note that $2 \sqrt{\Delta} t = 0 \in \pi_{13, 1} \overline{\TMF^{\mathcal{B}C_2}} / \tau^3$, since it is true mod $\tau$ (i.e.\ on the $E_2$ page), and there are no $\tau$ multiples in the bidegree. So we get a map of synthetic spectra $\Sigma^{13, 1} \TMF / 2 \to \overline{\TMF^{\mathcal{B}C_2}} / \tau^3$ picking out $\sqrt{\Delta} t$.

  From \Cref{fig:anss-tmf-mod-2}, we can read that $\nu^3 = \eta^2 \tilde{\nu^2} \in \pi_{9, 3}\TMF / 2$, noting that there are no $\tau$-divisible classes in that bidegree. Thus, $\nu^3 \sqrt{\Delta} t = \eta^2 \tilde{\nu^2}\sqrt{\Delta} t$. However, $\tilde{\nu^2} \sqrt{\Delta} t \in \pi_{20, 2}\overline{\TMF^{\mathcal{B}C_2}} / \tau^3 = 0$, which is a contradiction ($\tau^2 h_1^4 = 0$, so the $h_1$ towers cannot contribute).
\end{proof}

\begin{corollary}
  $d_3(\sqrt{\Delta} t) = x_{12, 4}$.
\end{corollary}

\begin{proof}
  This is equivalent to saying that $d_3(\sqrt{\Delta} t)$ is not $h_1$-divisible. If it were, then we can write
  \[
    d_3(\sqrt{\Delta} t) = h_1^4 y
  \]
  for some $y$. Then $h_1^4 d_3(y) = d_3(h_1^4 y) = 0$, so $d_3(y) = 0$ since $h_1^4$ is injective on the target bidegree. Then $d_3(h_1 v_1^2 y) = h_1^4 y$, and so $d_3(\sqrt{\Delta}t + h_1 v_1^2 y) = 0$.

  But the argument of \Cref{lemma:sqrt-delta-t} applies equally well to $\sqrt{\Delta} t + h_1 v_1^2 y$. So we get a contradiction.
\end{proof}

\begin{corollary}
  $d_3$ vanishes on any class in bidegree $(8k, 0)$.
\end{corollary}

\begin{proof}
  By the computation of the $E_2$ page, we can pick $x_{4, 0}$ and $x_{16, 0}$ to be such that they generate the $0$-line under $v_1^2$ and $\Delta^{\pm}$. Let $x_{8, 0} = v_1^2 x_{4, 0}$. Then $x_{8, 0}$ and $x_{16, 0}$ generate the classes in bidegree $(8k, 0)$ under $v_1^4$ and $\Delta^{\pm}$, and it suffices to show that $d_3(x_{8, 0}) = d_3(x_{16, 0}) = 0$.

  The strategy is to repeat the argument of \Cref{lemma:juggle} with $(d_1, h_1)$ replaced by $(d_3, h_1^3)$. To do so, we need to check that $d_3(x_{8, 0})$ and $d_3(x_{16, 0})$ are $O(v_1^4)$.

  We can compute that
  \[
    d_3(x_{8, 0}) = d_3(v_1^2 x_{4, 0}) = h_1^3 x_{4, 0} + v_1^2 h_1^2 t + O(v_1^4).
  \]
  But $h_1^3 x_{4, 0} = v_1^2 h_1^2 t + O(v_1^4)$ by \Cref{cor:cross} (the corollary shows that the remainder term is $O(a_1^3)$, but anything on the $E_2$ page that is $a_1^3$-divisible is also $a_1^4$-divisible). So this is $O(v_1^4)$.

  To see that $d_3(x_{16, 0}) = O(v_1^4)$, suppose instead that $d_3(x_{16, 0}) = h_1^2 \sqrt{\Delta} t + O(v_1^4)$. Then applying $d_3$ again gives $0 = h_1^2 x_{12, 4} + O(v_1^4)$, which is a contradiction since $h_1^2 x_{12, 4}$ is not $v_1^4$-divisible.
\end{proof}
This concludes the calculation of the $d_3$ differentials.

The $E_5$ page (with the $\ko$-like patterns omitted\footnote{It is easy to check that there cannot be differentials from the $\ko$-like classes since the possible targets are non-$g$-torsion}) is shown in \Cref{fig:anss-e5-page}. The differentials come from applying the Leibniz rule with $d_5(\Delta) = h_2 g$.

We then have $d_7$s in \Cref{fig:anss-e7-page} that are forced by the hidden $\nu$ extensions. Indeed, we have shown that the targets of the depicted differentials are $\nu$ times classes that are zero, hence must be hit by a differential, and the only possible sources are the ones we have drawn. A more careful argument using synthetic spectra can directly prove that these specific $d_7$'s must occur.

The $E_9$ page is then depicted in \Cref{fig:anss-e5-page}, which is still $\Delta^2$-invariant. We will show that the greyed out classes do not survive, while the black ones do. 

Afterwards, all the remaining differentials are long differentials that kill off high $\kappabar$ powers. These are shown in \Cref{fig:anss-long-d}, and the $E_\infty$ page is shown in \Cref{fig:anss-e-infty-page}. In the rest of the section, we shall show that the long differentials that occur are indeed what we indicated. We then conclude the calculation using that $\Delta^{\pm 8}$ is permanent.

\begin{sseqdata}[large, name = tmf anss, x range={0}{100}, y range={0}{20}]
  \foreach \delta in {0,...,5} {
    \foreach \g in {0,...,6} {
      \hook(\delta * 24 - 4 * \g, 4 * \g);
      \htwoedge(\delta * 24 + 12 - 4 * \g, 4 * \g);
    }
    \ifnum\delta=0
      \classoptions["t"](\delta * 24 + 1, 1)
    \else\ifnum\delta=1
      \classoptions["\Delta t"](\delta * 24 + 1, 1)
    \else
      \classoptions["\Delta^{\delta} t"](\delta * 24 + 1, 1)
    \fi
    \fi
  }
  \foreach \delta in {0,...,5} {
    \foreach \g in {0,...,6} {
      \IfExistsT(\delta * 24 + 22 - 4 * \g, 6 + 4 * \g) {
        \structline [dashed, bend left=30] (\delta * 24 + 19 - 4 * \g, 3 + 4 * \g)(\delta * 24 + 22 - 4 * \g, 6 + 4 * \g)
      }
    }
  }

  \foreach \delta in {0,...,3} {
    \foreach \g in {0,...,6} {
      \DrawIfValidDifferential5(\delta * 48 + 25 - 28 * \g, 1 + 4 * \g);
      \DrawIfValidDifferential5(\delta * 48 + 28 - 28 * \g, 2 + 4 * \g);
      \DrawIfValidDifferential5(\delta * 48 + 40 - 28 * \g, 2 + 4 * \g);
      \DrawIfValidDifferential7(\delta * 48 + 43 - 28 * \g, 3 + 4 * \g);
    }
  }
\end{sseqdata}

\begin{figure}
  \centering
  \begin{sideways}
    \printpage[name = tmf anss, page = 5]
  \end{sideways}
  \caption{$E_5$ page of the descent spectral sequence}\label{fig:anss-e5-page}
\end{figure}

\begin{figure}
  \centering
  \begin{sideways}
    \printpage[name = tmf anss, page = 7]
  \end{sideways}
  \caption{$E_7$ page of the descent spectral sequence}\label{fig:anss-e7-page}
\end{figure}

\begin{figure}
  \centering
  \begin{sideways}
    \begin{sseqpage}[name = tmf anss, page = 9]
      \def\labelscalefactor{0.65}
      \foreach \x in {0, 1} {
        \SseqParseInt\dp{2 * \x}
        \ifnum\x=0
          \def\deltafactor{}
        \else
          \def\deltafactor{\Delta^{\dp}}
        \fi

        \foreach \n in {0, ..., 2} {
          \SseqParseInt\wpow{3 + 4 * \n}
          \classoptions["w^{\wpow} \deltafactor t"](\x * 48 + 16 + \n * 20, 2 + \n * 4)
          \SseqParseInt\wpow{4 + 4 * \n}
          \classoptions["w^{\wpow} \deltafactor t"](\x * 48 + 21 + \n * 20, 5 + \n * 4)
          \SseqParseInt\wpow{5 + 4 * \n}
          \classoptions["w^{\wpow} \deltafactor t"](\x * 48 + 26 + \n * 20, 2 + \n * 4)
          \SseqParseInt\wpow{6 + 4 * \n}
          \classoptions["w^{\wpow} \deltafactor t"](\x * 48 + 31 + \n * 20, 3 + \n * 4)
        }
      }

      \foreach \x/\y in {0/18, 10/18, 15/19, 3/7, 8/10, 13/13, 18/10, 23/11, 28/14, 33/17, 38/14, 43/15, 48/18, 58/18, 63/19, 96/18, 97/1} {
        \classoptions[gray](\x, \y)
      }
      \classoptions["w^{15} t"]( 16 + 60, 2 + 12)
      \classoptions["w^{16} t"]( 21 + 60, 5 + 12)
      \classoptions["w^{17} t"]( 26 + 60, 2 + 12)
      \classoptions["w^{18} t"]( 31 + 60, 3 + 12)
      \classoptions["w^{19} t"]( 16 + 80, 2 + 16)
    \end{sseqpage}
  \end{sideways}
  \caption{$E_9$ page of the descent spectral sequence}\label{fig:anss-e9-page}
\end{figure}

\newcommand\LONGDMAXN{9}
\NewSseqGroup \longd {} {
  \hook
  \htwoedge(12, 0)
  \class(21, 5)
  \class(22, 6) \structline[hone]
  \structline [dashed, bend left=30] (19, 3)(22, 6)
  \foreach \n in {0, ..., \LONGDMAXN} {
    \class(26 + \n * 20, 2 + \n * 4)
    \class(31 + \n * 20, 3 + \n * 4)
    \class(36 + \n * 20, 6 + \n * 4)
    \class(41 + \n * 20, 9 + \n * 4)
  }
}
\begin{sseqdata}[large, name = tmf anss long, x range={88}{188}, y range={0}{40}, class label handler = { \def\result{\scalebox{0.65}{$#1$}} }]
  \foreach \x in {0, ..., 4} {
    \longd(\x * 48, 0)
    \SseqParseInt\dp{2 * \x}
    \ifnum\x=0
      \def\deltafactor{}
    \else
      \def\deltafactor{\Delta^{\dp}}
    \fi
    \classoptions["\deltafactor t"](\x * 48 + 1, 1)
    \foreach \n in {0, ..., \LONGDMAXN} {
      \SseqParseInt\wpow{3 + 4 * \n}
      \classoptions["w^{\wpow} \deltafactor t"](\x * 48 + 16 + \n * 20, 2 + \n * 4)
      \SseqParseInt\wpow{4 + 4 * \n}
      \classoptions["w^{\wpow} \deltafactor t"](\x * 48 + 21 + \n * 20, 5 + \n * 4)
      \SseqParseInt\wpow{5 + 4 * \n}
      \classoptions["w^{\wpow} \deltafactor t"](\x * 48 + 26 + \n * 20, 2 + \n * 4)
      \SseqParseInt\wpow{6 + 4 * \n}
      \classoptions["w^{\wpow} \deltafactor t"](\x * 48 + 31 + \n * 20, 3 + \n * 4)
    }
  }
  \foreach \m in {0, 1} {
    \d17(97 + \m * 48, 1)
    \d17(112 + \m * 48, 2)
    \d17(117 + \m * 48, 5)
    \foreach \n in {0, ..., \LONGDMAXN} {
      \DrawIfValidDifferential23(122 + \n * 20 + \m * 48, 2 + \n * 4)
      \DrawIfValidDifferential19(127 + \n * 20 + \m * 48, 3 + \n * 4)
      \DrawIfValidDifferential17(132 + \n * 20 + \m * 48, 6 + \n * 4)
      \DrawIfValidDifferential17(137 + \n * 20 + \m * 48, 9 + \n * 4)
    }
  }
\end{sseqdata}

\begin{figure}
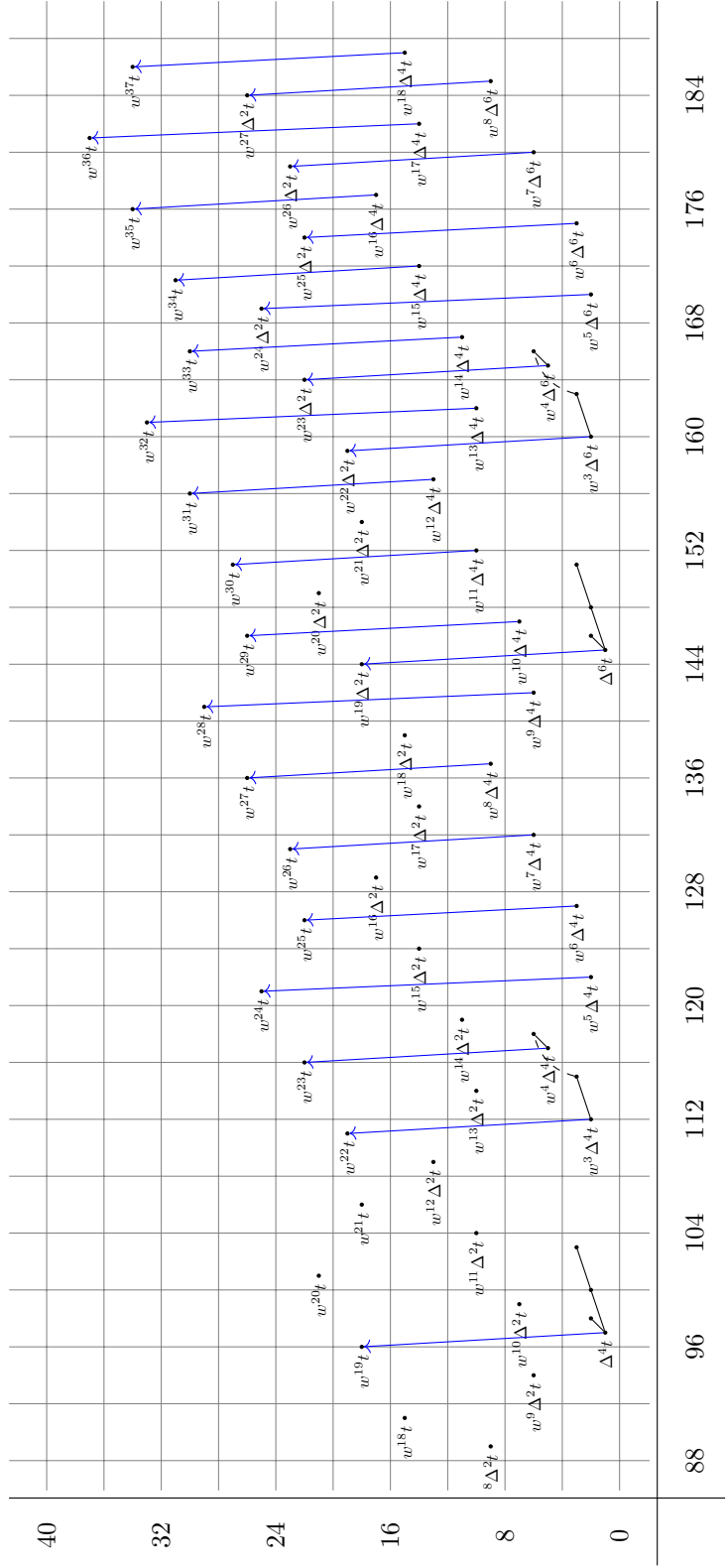

  \centering
  \begin{sideways}
    \printpage[name = tmf anss long]
  \end{sideways}
  \caption{Remaining long differentials}\label{fig:anss-long-d}
\end{figure}

\begin{figure}
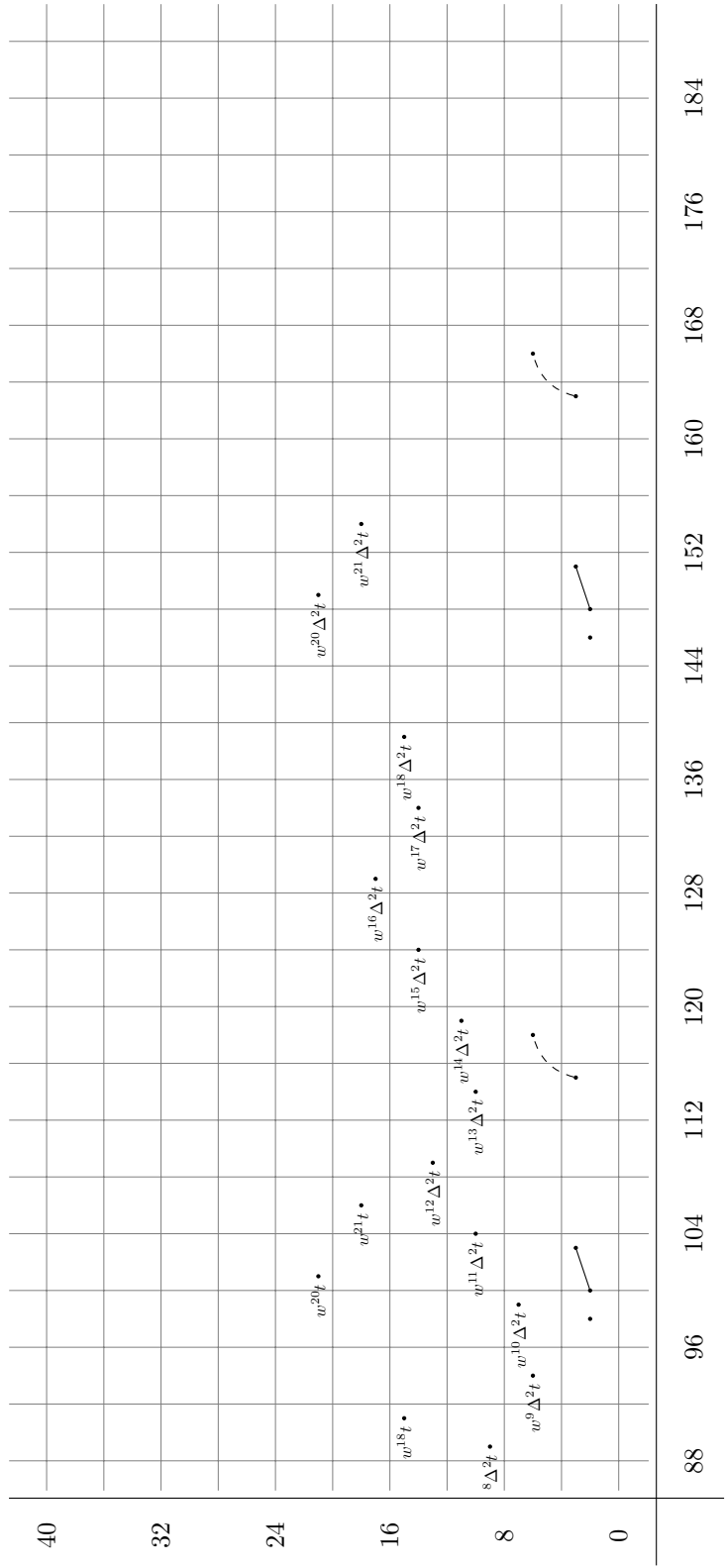

  \centering
  \begin{sideways}
    \printpage[name = tmf anss long, page = \infty]
  \end{sideways}
  \caption{The $E_\infty$ page of the descent spectral sequence}\label{fig:anss-e-infty-page}
\end{figure}

We start with the observation that
\begin{lemma}
  There are no classes above the $s = 24$-line that survive.
\end{lemma}

\begin{proof}
  On the $E_4$ page, multiplication by $g$ is injective and surjective above the $3$-line. If $x$ is an element above the $24$-line, then we can write $x = g^6 y$ for some $y$. If $x$ does not support any differential, then neither does $y$, by injectivity of $g$-multiplication. So $y$ is permanent, and the class representing $x$ is $\kappabar^6$-divisible. But $\kappabar^6 = 0$ in $\tmf$.
\end{proof}

The hardest part is to show that $\Delta^2 t$ is permanent. The difficulty here is again translational invariance. Our starting piece of knowledge is that $t$ is permanent, and we want to somehow deduce that $\Delta^2 t$ is permanent too. However, we must not allow ourselves to repeat the argument, using that $\Delta^2 t$ is permanent to deduce that $\Delta^4 t$ is, because it is not.

The key property we can make use of is the fact that the class $t$ extends to a map from $\TMF \otimes \RP^\infty$. Our job would be easy if $\Delta^2$ of the bottom cell is permanent in $\TMF \otimes \RP^\infty$, but that's not true. However, we can get by with the following version:

\begin{lemma}
  The class $t\colon \Sigma \TMF \to \overline{\TMF^{\mathcal{B}C_2}}$ extends to a map from $\TMF \otimes L$.
\end{lemma}
Recall that $L$ is the dual of $DL$, as in the statement of the main theorem (\Cref{thm:main}). Its cell diagram is depicted in \Cref{fig:cell-l}.

\begin{proof}
  First of all, it extends to $\Sigma \TMF / 2$ since $2 t = 0$. The obstruction to extending to the $4$-cell is $\langle \eta, 2, t\rangle$. Since $t$ comes from restricting the norm map $\TMF_{h C_2} \to \TMF^{\mathcal{B}C_2}$, we know it extends to a map $\RP^4 \to \overline{\TMF^{\mathcal{B}C_2}}$. Let $y \in \pi_* \overline{\TMF^{\mathcal{B}C_2}}$ be the image of the $3$-cell. Then the cell structure of $\RP^4$ (\Cref{fig:cell-l}) tells us
  \[
    \langle \eta, 2, t \rangle = 2y.
  \]
  But all possible images of $y$ are $2$-torsion. So $\langle \eta, 2, t \rangle = 0$. Finally, the obstruction to extending to all of $L$ is $\langle \nu, \eta, 2, t \rangle$, which is defined since $\langle \nu, \eta, 2\rangle = 0$ with no indeterminacy. However, the only possible class is a $\nu$-multiple, hence is in the indeterminacy. So $0 \in \langle \nu, \eta, 2, t \rangle$, and we can extend to $L$.
\end{proof}

\begin{figure}[h]
  \centering
  \begin{tikzpicture}[scale=0.5]
    \begin{celldiagram}
      \two{1}
      \eta{2}
      \nu{4}
      \n{8} \n{4} \n{2} \n{1}

      \foreach \y in {1,2,4,8} {
        \node [left] at (-0.8, \y) {$w_\y$};
      }
    \end{celldiagram}
    \node [right] at (1, 6) {$\nu$};
    \node [right] at (0.5, 3) {$\eta$};
    \node [right] at (0, 1.5) {$2$};

    \node at (0, 0) {$L$};

    \begin{scope}[shift={(5, 0)}]
      \begin{celldiagram}
        \two{1}
        \two{3}
        \eta{2}
        \n{1} \n{2} \n{3} \n{4}
      \end{celldiagram}
      \node at (0, 0) {$\RP^4$};
    \end{scope}
  \end{tikzpicture}
  \caption{Cell diagrams of $L$ and $\RP^4$}\label{fig:cell-l}
\end{figure}
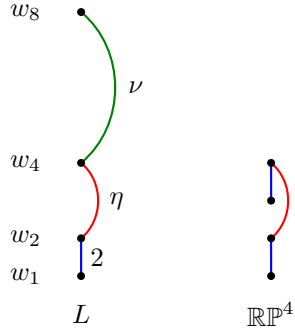

\begin{remark}
  \emph{A posteriori}, we expect such a map to exist. We know that $\overline{\TMF^{\mathcal{B}C_2}} = \TMF \otimes DL$, and this is the map $\TMF \otimes L \to \TMF \otimes DL$ whose cofiber is $\KO$.
\end{remark}

Let $w_k$ be the $k$-cell of $L$.
\begin{theorem}
  In the Adams--Novikov spectral sequence of $\TMF \otimes L$, the class $\Delta^2 w_1$ survives and has order $2$.
\end{theorem}

\begin{proof}
  It suffices to prove this for $\tmf \otimes L$ instead.

  We do not know of a way to compute the $E_2$ page of the Adams--Novikov spectral sequence for $\tmf \otimes L$, as the attaching map of the $4$-cell is filtration $0$ but non-injective in homology, so there is no long exact sequence. To remedy this problem, we use a modified Adams spectral sequence via the technology of synthetic spectra.

  First observe that $2\colon \tmf \to \tmf$ is in fact injective in $BP$-homology, since $BP \otimes \tmf$ is non-torsion (see e.g.\ \cite[Corollary 5.2]{homology-tmf}). So $\nu (\tmf / 2) = \nu (\tmf) / 2$.

  We next construct synthetic $\nu(\tmf)$-modules $\widetilde{Q}, \widetilde{L}$ by the cofiber sequences
  \[
    \begin{tikzcd}[row sep=tiny]
      \Sigma^{3, 0}\tmf \ar[r, "{[\eta w_2]}"] & \nu(\Sigma \tmf / 2) \ar[r] & \widetilde{Q}\\
      \Sigma^{7, 0}\tmf \ar[r, "{[\nu w_4]}"] & \widetilde{Q} \ar[r] & \widetilde{L}
    \end{tikzcd}
  \]
  where $[\eta w_2]$ and $[\nu w_4]$ are the attaching maps of the $4$- and $8$-cell respectively. They are characterized by the fact that the projection onto the top cell of their targets are $\eta$ and $\nu$ respectively.

  The universal property of the cofiber gives us natural comparison maps $\widetilde{Q} \to \nu (\tmf \otimes Q)$ and $\widetilde{L} \to \nu (\tmf \otimes L)$. For example, the second map is obtained from the first via
  \[
    \begin{tikzcd}
      \Sigma^{7, 0} \tmf \ar[r, "{[\nu w_4]}"] \ar[d, equals] & \tilde{Q} \ar[r] \ar[d] & \widetilde{L} \ar[d, dashed]\\
      \Sigma^{7, 0} \tmf = \nu(\Sigma^7 \tmf) \ar[r, "{\nu([\nu w_4])}"] & \nu (\tmf \otimes Q) \ar[r] & \nu (\tmf \otimes L).
    \end{tikzcd}
  \]
  Here the top row is a cofiber sequence in the category of synthetic spectra, and the bottom row is $\nu$ applied to a cofiber sequence in the category of spectra. The only thing to check is that the left-hand square commutes, which is true since every map $S^{k, 0} \to \nu Z$ is uniquely of the form $\nu f$; there are no $\tau$-torsion classes in this bidegree since these would have to be hit by a differential from below the $0$-line.

  Given this, it suffices to show that $\Delta^2 w_1$ survives in $\widetilde{L}$. To understand $\widetilde{L}$, we start with $\Sigma \tmf / 2$, whose ANSS was computed by \cite{tmf-rp2} and is shown in \Cref{fig:anss-c2-full} (with $\ko$-like terms omitted as usual). The important feature is that the differential on $\Delta^2 w_1$ hits $g^2 x w_1$ --- we have to show that this class vanishes in $\widetilde{L}$.

  The synthetic cofiber sequence $\Sigma^{3, 0} \tmf \to \nu (\Sigma \tmf / 2) \to \widetilde{Q}$ tells us the $E_2$ page of the ANSS for $\tilde{Q}$ sits in a long exact sequence between that of $\Sigma^{3, 0} \tmf$ and $\nu (\Sigma \tmf / 2)$. This is displayed in \Cref{fig:anss-q-full}, where the depicted differentials are the connecting map. The crucial claim in this diagram is that there is a hidden extension $\eta [\nu w_4] = x w_1$ on the $E_2$ page. Then since $\widetilde{L}$ is obtained by killing $[\nu w_4]$, we know that $gxw_1 = 0$ in the $E_2$ page of $\tilde{L}$. Since there are no higher filtration terms, $\Delta^2 w_1$ must survive.

  To see this hidden extension, note that $x \in \pi_* \tmf / 2$ detects $\nu^2$ on the top cell. If we quotient out the bottom cell in $\tilde{Q}$, then we can write the class of interest as
  \[
    \eta [\nu w_4] = \eta \langle \nu, \eta, w_2 \rangle = \langle \eta, \nu, \eta\rangle w_2 = \nu^2 w_2,
  \]
  as desired. In this equation, $w_2$ is the homotopy class of the $2$-cell of $\tilde{Q}/w_1$, which is now an actual element since we killed off the bottom cell.

  Finally, it is straightforward to check that there are no classes above $\Delta^2 w_1$, so it must have order $2$.
\end{proof}

\DeclareSseqGroup\tmfAnssModTwoUnitPrime {} {
  \class (0, 0) \savestack \hone \hone
  \class (2, 0) \structline[dashed] \hone \hone
  \restorestack
  
  \htwo \htwo \htwo
  \honei \honei
  \htwo
  \honei \structline[dashed](9, 3)

  \class (15, 3)
  \honei
  \htwo
  \honei \honei
  \structline[dashed](15, 3)
  \htwo \htwo
}

\DeclareSseqGroup\tmfAnssZero {} {
  \class[rectangle, fill=white](0, 0) \hone \hone \hone
  \class[circlen=2](3, 1) \structline[dashed] \structline(0, 0) \htwo \htwo \honei
}

\DeclareSseqGroup\tmfAnssHigh {} {
  \class(-5, -1) \honei \htwo \class[circlen = 3](0, 0) \structline \hone \hone \hone
  \class[circlen=2](3, 1) \structline[dashed] \structline(0, 0) \htwo \htwo \honei
}

\begin{figure}
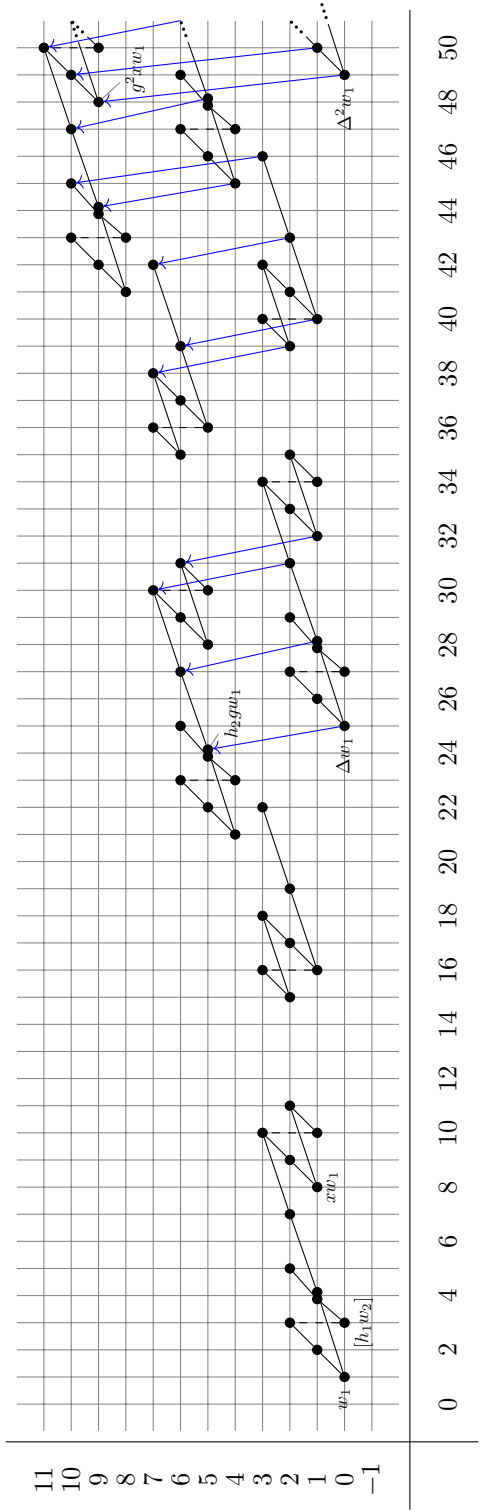

  \def\labelscalefactor{0.7}
  \centering
  \begin{sideways}
    \begin{sseqpage}[small, scale = 1.016, x range={0}{50}, y range={-1}{11}]
      \tmfAnssModTwoUnitPrime(1, 0)
      \tmfAnssModTwoUnitPrime(21, 4)
      \tmfAnssModTwoUnitPrime(41, 8)
      \tmfAnssModTwoUnitPrime(25, 0)
      \tmfAnssModTwoUnitPrime(45, 4)
      \tmfAnssModTwoUnitPrime(49, 0)

      \foreach \n in {0, 1} {
        \d5(25 + 20 * \n, 0 + 4 * \n)(24 + 20 * \n, 5 + 4 * \n, 2)
        \d5(28 + 20 * \n, 1 + 4 * \n, 2)
        \d5(31 + 20 * \n, 2 + 4 * \n)
        \d5(32 + 20 * \n, 1 + 4 * \n)

        \d5(39 + 20 * \n, 2 + 4 * \n)
        \d5(40 + 20 * \n, 1 + 4 * \n)
        \d5(43 + 20 * \n, 2 + 4 * \n)
      }

      \d7(46, 3)
      \d9(49, 0)
      \d9(50, 1)

      \classoptions["w_1"](1, 0)
      \classoptions["{[h_1 w_2]}" {below=-0.5pt}](3, 0)
      \classoptions["x w_1" {below=-0.5pt}](8, 1)
      \classoptions["\Delta w_1"](25, 0)
      \classoptions["\Delta^2 w_1"](49, 0)
      \classoptions["h_2 g w_1" { anchor = north west, pin=gray, yshift=-4, xshift=4 }](24, 5, 2)
      \classoptions["g^2 x w_1" { anchor = north west, pin=gray , yshift=-8, xshift=3 }](48, 9)
    \end{sseqpage}
  \end{sideways}
  \caption{Adams--Novikov spectral sequence for $\tmf / 2$}\label{fig:anss-c2-full}
\end{figure}

\begin{figure}
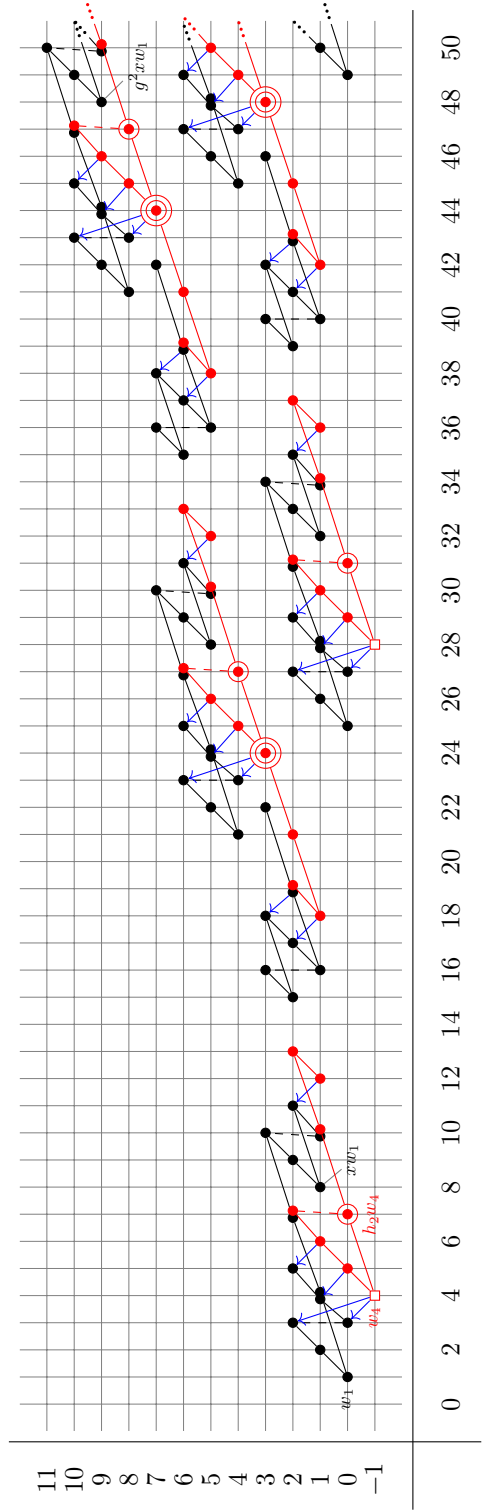

  \def\labelscalefactor{0.7}
  \centering
  \begin{sideways}
    \begin{sseqpage}[small, scale = 1.016, x range={0}{50}]
      \tmfAnssModTwoUnitPrime(1, 0)
      \tmfAnssModTwoUnitPrime(21, 4)
      \tmfAnssModTwoUnitPrime(41, 8)
      \tmfAnssModTwoUnitPrime(25, 0)
      \tmfAnssModTwoUnitPrime(45, 4)
      \tmfAnssModTwoUnitPrime(49, 0)

      \tmfAnssZero[red](4, -1)
      \tmfAnssHigh[red](24, 3)
      \tmfAnssHigh[red](44, 7)

      \tmfAnssZero[red](28, -1)
      \tmfAnssHigh[red](48, 3)

      \tmfAnssZero[red](52, -1)

      \foreach \n in {0, 1, 2} {
        \pgfmathsetmacro\endk{2 - \n}
        \foreach \k in {0, ..., \endk} {
          \d1(4 + 24 * \n + 20 * \k, -1 + 4 * \k)
          \replacesource["2"]
          \d3(4 + 24 * \n + 20 * \k, -1 + 4 * \k)
          \d1(5 + 24 * \n + 20 * \k, 0 + 4 * \k)(4 + 24 * \n + 20 * \k, 1 + 4 * \k, 1)
          \d1(6 + 24 * \n + 20 * \k, 1 + 4 * \k)
          \d1(12 + 24 * \n + 20 * \k, 1 + 4 * \k)
          \ifnum\k>0 
            \d1(-2 + 24 * \n + 20 * \k, -3 + 4 * \k)
            \d1(-1 + 24 * \n + 20 * \k, -2 + 4 * \k, 1)
          \fi
        }
      }
      \classoptions["w_1"](1, 0)
      \classoptions["w_4"](4, -1)
      \classoptions["h_2 w_4" {below=-0.5pt}](7, 0)
      \classoptions["x w_1" { anchor = north west, pin=gray , yshift=-8, xshift=4}](8, 1)
      \classoptions["g^2 x w_1" { anchor = north west, pin=gray , yshift=-8, xshift=4}](48, 9)
    \end{sseqpage}
  \end{sideways}
  \caption{Adams--Novikov spectral sequence for $\widetilde{Q}$}\label{fig:anss-q-full}
\end{figure}

\begin{corollary}
  The class $\Delta^2 t$ in the descent spectral sequence of $\overline{\TMF^{\mathcal{B}C_2}}$ is permanent and has order $2$.
\end{corollary}

\begin{proof}
  We previously constructed a map $\TMF \otimes L \to \overline{\TMF^{\mathcal{B}C_2}}$ where the bottom cell hits $t$. Applying $\nu$ to this, we get a map $\tilde{t}\colon \nu(\TMF \otimes L) \to \nu (\overline{\TMF^{\mathcal{B}C_2}})$ where the bottom cell hits $\tau t$.\footnote{We use $t$ to denote the lift of the permanent class $x_{1, 1} \in \pi_{1, 1} \nu(\overline{\TMF^{\mathcal{B}C_2}}) / \tau$ to $\pi_{1, 1} \nu(\overline{\TMF^{\mathcal{B}C_2}})$. Since the bottom cell hits $\pi_{1, 0} \nu(\overline{\TMF^{\mathcal{B} C_2}})$ and is equal to $t$ after $\tau$-inversion, it hits $\tau t$.} Now consider $\tilde{t}(\Delta^2 w_1)$. This is a permanent class, and since $\Delta^2 \in \TMF / \tau^2$, after modding out by $\tau^2$, we know that it must hit $\tau \Delta^2 t$. So $\tilde{t}(\Delta^2 w_1)$ is detected by $\Delta^2 t$.
\end{proof}

For the rest of the section, let $z = t$ or $\Delta^2 t$. It remains to consider the ``$w$-chains'' starting from $z$. There is a partially defined multiplication-by-$w$ operation on the $E_9$ page, where $w$ increases stem by $5$. To formally define this, we set
\[
  w^3 z = h_2 \sqrt{\Delta} z, \quad w^5 z = \Delta h_1 z,\quad w^6 z = \Delta h_2^2 z,\quad w^{k + 4} z = g w^k z.
\]
\begin{corollary}
  The $w$ chain starting from $z$ is permanent.
\end{corollary}

\begin{proof}
  The argument of \Cref{cor:hidden-ext} shows that $w^3 z$ is permanent and represents $\tilde{\kappa} z$ (for $\Delta^2 t$, run the same argument but start with the map $\Sigma^{49} \TMF / 2 \to \overline{\TMF^{\mathcal{B}C_2}}$ detecting $\Delta^2 t$, which we now know to be permanent and $2$-torsion). Since $w^5 = [h_1 \Delta]$ is permanent, we know that $w^5 z$ is also permanent.

  This leaves the $w^6$ terms, before we can conclude by $g = w^4$-periodicity. The crucial observation is that $\nu^3 z = 0$, and then Moss' convergence theorem \cite{moss} tells us $w^6 z$ detects $\langle \kappabar, \nu^3, z \rangle$ and is permanent.

  To see that $\nu^3 z = 0$, it suffices to show that $\nu^3 w_1 = \nu^3 \Delta^2 w_1 = 0$ in $\tilde{L}$. We have previously seen that in $\tilde{Q}$, we have $\nu^3 w_1 = \eta^2 x w_1 = \eta^3 [\nu w_4]$. Since $[\nu w_4]$ is killed in $\tilde{L}$, so is $\nu^3 w_1$. Since $\Delta^2$ exists on the $E_2$ page, we know that $\nu^3 \Delta^2 w_1 = 0$ on the Adams--Novikov $E_2$ page of $\tilde{L}$, and there are no higher filtration classes, so the product must be zero.
\end{proof}

\begin{corollary}
  There is a differential $d_?(\Delta^4 w^k z) = w^{k + 19} z$ for all $k \geq 3$ and $k = 0$.
\end{corollary}
Here the length of the differential depends on the value of $k \pmod 4$, which can be read off the charts. The precise values are, however, unimportant.

\begin{proof}
  For $k \geq 8$, this follows from $\kappabar^6 = 0$, since these are the only classes that can hit them. For smaller $k$, this follows from $g$-division.
\end{proof}

\section{Identification of the last factor}\label{section:ident}
To identify $\overline{\TMF^{\mathcal{B}C_2}} \cong \TMF \otimes DL$, we map $DL$ in by obstruction theory, and show it is an isomorphism after base change to $\TMF_1(3)$. To do so, we need to understand the $\TMF_1(3)$-homology of $DL$.
\begin{lemma}
  We can choose classes $y_{-8}, y_{-4}, y_{-2} \in \pi_* \TMF_1(3) \otimes DL$ such that
  \begin{enumerate}
    \item $y_{-k} \in \pi_{-k} \TMF_1(3) \otimes DL$;
    \item $y_{-8}$ is the bottom cell of $DL$;
    \item $\{y_{-4}, y_{-2}\}$ generates $\pi_* \TMF_1(3) \otimes DL$ as a free $\pi_* \TMF_1(3)$-module; and
    \item 
      \[
        \begin{aligned}
          y_{-8} &= v_2^{-1} (a_1 y_{-4} + 2 y_{-2}) + O(2^2),\\
          d(y_{-4}) &\equiv \psi(y_{-4}) - [1] y_{-4} = [r] y_{-8} + O(2^2),
        \end{aligned}
      \]
      where $v_2 = a_1^3 - 27 a_3$.
  \end{enumerate}
\end{lemma}
The choice of $y_{-4}$ and $y_{-2}$ is pretty much arbitrary. Other choices will result in slightly different formulas. These are chosen to simplify the ensuing calculation.

\begin{remark}
  Note that $[r]$ is the class that represents $h_2$, not $s^2$; the latter is not a cocycle since $2 \not= 0$.
\end{remark}

\begin{proof}
  We carefully construct $\tmf \otimes DL$ in the category of $\tmf$-modules. We start with the bottom cell and attach $y_{-4}$ to kill $[r] y_{-8}$. The class $y_{-4}$ is only well-defined up to multiples of $y_{-8}$, which in this case is integral multiples of $a_1^2 y_{-8}$. The coboundary of $a_1^2 y_{-8}$ is $[12r] y_{-8}$, and so
  \[
    \psi(y_{-4}) = y_{-4} + [k r] y_{-8},
  \]
  where $k \equiv 1 \pmod 4$.

  The next cell will kill of the cocycle
  \[
    \{h_1 y_{-4}\} = [s] y_{-4} - [k(a_1 r + t)] y_{-8} - ? ([s]a_1^2 - [12t]) y_{-8}.
  \]
  Here $?$ is either $1$ or $0$, noting that twice the class is a coboundary. There is exactly one choice of $?$ for which this cocycle is permanent, since there is a $d_3([s] a_1^2 - [12t]) = h_1^4$. On the other hand, we know that $y_{-4}$ is not entirely well-defined, and we can absorb the term into $y_{-4}$ (and redefine $k$). We choose $y_{-4}$ so that $? = 1$.

  We now set
  \[
    \psi(y_{-2}) = y_{-2} + \{h_1 y_{-4}\}.
  \]
  and then the class
  \[
    2 y_{-2} + a_1 y_{-4} - k a_3 y_{-8} - (a_1^3 - 12 a_3) y_{-8}.
  \]
  is a cocycle, which the top cell kills off. So we get the relation
  \[
    2 y_{-2} + a_1 y_{-4} - (a_1^3 - 27 a_3) y_{-8} + O(2^2),
  \]
  as desired.
\end{proof}

We now construct a map $f\colon \TMF \otimes DL \to \overline{\TMF^{\mathcal{B}C_2}}$. This is constructed via obstruction theory. The relevant homotopy groups are in the range $[-8, 0]$, which are depicted in \Cref{fig:homotopy-neg}. In this range, all the homotopy groups come from the $\ko$-like patterns.
\begin{figure}[h]
  \centering
  \sseqnewclasspattern{horizontal}{
    (0,0);
    (0,0)(0,0);
    (0,0)(0,0)(0,0);
    (0.45,0)(0.15,0)(-0.15,0)(-0.45,0);
    (0.45,0)(0.15,0)(-0.15,0)(-0.45,0)(0.75,0);
  }
  \begin{sseqpage}[small, x range={-12}{7}, class pattern = horizontal]
    \foreach \n in {1, ..., 4} {
      \pgfmathsetmacro\opac{\n * 0.25}
      \begin{scope}[opacity = \opac]
        \class[rectangle, fill=white](-12, 0)
        \class[rectangle, fill=white](-8, 0)\hone \hone
        \class[rectangle, fill=white](-4, 0)
        \class[rectangle, fill=white](0, 0)\hone \hone
        \class[rectangle, fill=white](4, 0)
      \end{scope}
    }

    \class(2, 2) \honei \htwo \htwo
  \end{sseqpage}
  \caption{Homotopy groups of $\overline{\TMF^{\mathcal{B}C_2}}$ with $\ko$-like terms}\label{fig:homotopy-neg}
\end{figure}

In general, let $z_{-k}$ be the images of $y_{-k}$ under $f$ (after base change to $\TMF_1(3)$). In constructing $f$, the first step is to pick the image of the bottom cell, i.e.\ the value of $z_{-8}$. This lives in $\pi_{-8}(\overline{\TMF^{\mathcal{B}C_2}})$, which is the direct sum of infinitely many copies of $\Z$. Choosing $z_{-8}$ requires a bit of care, but once we have chosen it, we can always extend it to a full map $f$. Indeed, the obstructions are
\[
  \nu x_{-8, 0},\quad \langle \eta, \nu, x_{-8, 0}\rangle,\quad \langle 2, \eta, \nu, x_{-8, 0} \rangle,
\]
which all vanish because
\[
  \pi_{-5} \overline{\TMF^{\mathcal{B}C_2}} = \pi_{-3} \overline{\TMF^{\mathcal{B}C_2}} = \pi_{-2} \overline{\TMF^{\mathcal{B}C_2}} = 0.
\]

While the extension always exists, it is not unique. Specifically, the extension to the second cell is not unique. This results in $z_{-4}$ being well-defined up to a permanent class in $\pi_{-4} \overline{\TMF^{\mathcal{B}C_2}}$, which is again infinitely many copies of $\Z$. After extending to the second cell, there is a unique way to extend all the way to $DL$, because $\pi_{-2} \overline{\TMF^{\mathcal{B}C_2}} = \pi_{-1} \overline{\TMF^{\mathcal{B}C_2}} = 0$.

We are now ready to choose $z_{-8}$. We definitely don't want this to be $v_1^4$-divisible, but this only defines $z_{-8}$ up to $v_1^4$-multiples. After some experimentation, we settle on the following choice:
\begin{lemma}\label[lemma]{lemma:computer-z-8}
  In the $E_2$ page of the Adams--Novikov spectral sequence of $\overline{\TMF^{\mathcal{B}C_2}}$, there is a permanent cocycle of the form
  \[
    z_{-8} = v_2^{-1}[a_1(z^2 - a_3^{-1} a_1^2 z) + 2z] + O(2^2)
  \]
  This will be our choice of $z_{-8}$.
\end{lemma}
This is a lift of $\Delta^{-1} a_1 x_{14, 0}$.
\begin{proof}
  Computer calculation (see \Cref{section:sage}) verifies that this is a cocycle mod $2^2$. In the bidegree $(-8, 0)$, there are no $2$-Bockstein differentials and DSS differentials, so any cocycle mod $2^n$ would lift to a permanent cocycle.
\end{proof}

Given the relation between the $y_{-k}$, we know there must be some $g$ such that
\[
  \begin{aligned}
    z_{-4} &= z^2 - a_3^{-1} a_1^2 z + 2g + O(2^2)\\
    z_{-2} &= z + a_1 g + O(2).
  \end{aligned}
\]
The indeterminacy tells us $2g$ is well-defined up to a permanent class.

\begin{lemma}\label[lemma]{lemma:computer-massey}
  There is a choice of $f$ such that $g = 0 \bmod 2$.
\end{lemma}

\begin{proof}
  In $\TMF_1(3) \otimes DL$, we have a cobar differential
  \[
    d(y_{-4}) = [r] y_{-8} + O(2^2).
  \]
  So we find that
  \[
    d(z_{-4}) = [r] z_{-8} + O(2^2).
  \]
  On the other hand, by computer calculation, we find that (see \Cref{section:sage})
  \[
    d(z^2 - a_1^2 a_3^{-1} z) = [r] z_{-8} + O(2^2).
  \]
  Comparing the two, we must have
  \[
    d(2g) = O(2^2).
  \]
  So $g$ is a cocycle mod $2$. Since there are no $2$-Bocksteins, $g$ lifts to an actual cocycle $\tilde{g}$, and $2\tilde{g}$ is permanent. So we can subtract $2\tilde{g}$ from $z_{-4}$ so that $g$ is now $0 \bmod 2$.
\end{proof}

\begin{corollary}
  This choice of $f$ is an equivalence, and hence $\overline{\TMF^{\mathcal{B}C_2}} \cong \TMF \otimes DL$.
\end{corollary}

\begin{proof}
  It suffices to show that $f$ is an equivalence after base change to $\TMF_1(3)$. Moreover, since we $2$-complete, we can further reduce mod $2$. Both surjectivity and injectivity in $\pi_*$ are easy linear algebra.
\end{proof}

\section{Further questions}\label{section:further-questions}
\begin{itemize}
  \item Is there a ``geometric'' description of the piece $\TMF \otimes DL$, similar to the representation-theoretic description of $\KO^{\mathcal{B}G}$? Can such a description streamline the calculations in the paper? (e.g.\ avoid the need of computers and ``explain'' the formula for $z_{-8}$)

    In particular, the equation $-v_3 y_{-8} + a_1 y_{-4} + 2 y_{-2}$ in $\TMF_1(3) \otimes DL$ looks remarkably similar to the defining equation $a_3 z^4 - a_1 z^2 + 2 z$ in $\TMF_1(3)^{\mathcal{B}C_2}$.

  \item The initial calculation of the Hopf algebroid is made possible by the fact that the inversion map of a Weierstrass elliptic curve has a simple formula. For $C_3$, we can take the equalizer of the inversion map and the duplication map $[2]$, but the latter is unwieldy.
\end{itemize}

\appendix
\section{Connective \texorpdfstring{$C_2$}{C2}-equivariant \texorpdfstring{$\tmf$}{tmf}}\label{section:connective}
At the prime $2$, one prized property of $\tmf$ is
\[
  H_* \tmf = \A \square_{A(2)} \F_2,
\]
which lets us identify the $E_2$-page of its Adams spectral sequence as $\Ext_{\A(2)}(\F_2, \F_2)$.

It is natural to expect $C_2$-equivariant $\tmf$ to have a similar properties, and there have been attempts to construct $C_2$-equivariant $\tmf$ along these lines (e.g.\ \cite{ricka-mmf}). The goal of this section is not to construct $C_2$-equivariant $\tmf$, but to deduce properties of any such construction based on its homology.

\begin{theorem}
  Let $\tmf_{C_2} \in \Sp_{C_2}$ be a spectrum such that
  \[
    (\HF_2)_\star \tmf_{C_2} = \aesq{2} \Mt
  \]
  and the $C_2$ action on the underlying spectrum $\iota \tmf_{C_2}$ is trivial. Then
  \[
    H_* (\tmf_{C_2})^{C_2} \cong H_*(\tmf \oplus \tmf \oplus \tmf \otimes L)
  \]
  as a Steenrod comodule. If $(\tmf_{C_2})^{C_2}$ has a $\tmf$-module structure, then we have an isomorphism of 2-completed $\tmf$-modules.
\end{theorem}
Note that it is $L$ that appears here, not $DL$, so $\Delta^{-1} (\tmf_{C_2})^{C_2}$ and $\TMF^{C_2}$ are duals as $\TMF$-modules (after completion).

\begin{remark}
  We would expect that the underlying spectrum of $\tmf_{C_2}$ is $\tmf$ with the trivial action, and that there is a natural ring map $\tmf \to (\tmf_{C_2})^{C_2}$, giving the target a $\tmf$-module structure. The proof uses weaker (albeit less natural) assumptions, which is what we state in the theorem.
\end{remark}

Since all we know about this hypothetical $\tmf_{C_2}$ is its homology, it will be cleaner to work with a general $C_2$-spectrum whose homology is $\aesq{n} \Mt$ for some $n$, specializing to $n = 2$ only at the very end. Nevertheless, we shall shortly see that such a spectrum can only exist for $n \leq 2$, as in the non-equivariant case.

\subsection*{Conventions}
\begin{itemize}
  \item For consistency with existing literature, we use $(-)^{C_2}$ to denote the genuine fixed points of a genuine $C_2$-spectrum, instead of the previous $(-)^{\mathcal{B}C_2}$.

  \item We work with $\RO(C_2)$-graded homotopy groups throughout. We have an explicit identification $\RO(C_2) = \Z[\sigma]/(\sigma^2 - 1)$, and we shall write $p + q\sigma = (p + q, q)$ --- the first degree is the total degree of the representation and the second degree is the number of $\sigma$'s, also called the \emph{weight}. We use $\star$ to denote an $\RO(C_2)$-grading and $*$ for an integer grading.

  \item We let $\rho \colon S^0 \to S^\sigma$ be the natural inclusion of the representation spheres (more commonly known as $a$). Then $\rho \in \pi_{-1, -1}(S^0)$.

  \item We write $\iota\colon  \Sp_{C_2} \to \Sp$ for the underlying spectrum functor.
\end{itemize}

\subsection{\texorpdfstring{$C_2$}{C2}-equivariant homotopy theory}
Let $X$ be a $C_2$-spectrum. The main difficulty in analyzing $H_* X^{C_2}$ is that taking categorical fixed points is not symmetric monoidal, so there is no \emph{a priori} relation between $H\F_2 \otimes X^{C_2}$ and $(\HF_2 \otimes X)^{C_2}$. To understand $X^{C_2}$, we consider two other functors $\Sp_{C_2} \to \Sp$ that \emph{are} symmetric monoidal, namely the geometric fixed points $X^{\Phi C_2}$ and the underlying spectrum $\iota X$. The goal of this section is to understand the effects of these constructions on $\RO(C_2)$-graded \emph{homotopy} groups.

The homotopy groups of the categorical fixed points are easy to compute. Essentially by definition, we have
\[
  \pi_* X^{C_2} = \pi_{*, 0} X.
\]

The underlying spectrum and geometric fixed points are obtained by performing certain constructions in $\Sp_{C_2}$ and then applying categorical fixed points.

In the $C_2$-equivariant world, the geometric fixed points admit a particularly simple description, since $\rho^{-1}S^0$ is a model of $\tilde{E}C_2$:
\begin{lemma}\label[lemma]{lemma:geometric-homotopy}
  The geometric fixed points is given by
  \[
    X^{\Phi C_2} = (\rho^{-1} X)^{C_2},
  \]
  and we can write its homotopy groups as
  \[
    \pi_* X^{\Phi C_2} = \pi_{*, 0} (\rho^{-1} X) = \pi_\star X / (\rho - 1),
  \]
  where in the latter formulation, an element of bidegree $(s, w)$ in $\pi_\star X$ gives an element of degree $s - w$ in $\pi_* X^{\Phi C_2}$.
\end{lemma}

The underlying spectrum is slightly more involved. It can be expressed as 
\[
  \iota X = ((C_2)_+ \otimes X)^{C_2} = F((C_2)_+, X)^{C_2},
\]
since $(C_2)_+$ is self-dual. To compute its homotopy groups, we use the cofiber sequence
\[
  (C_2)_+ \to S^0 \overset\rho\to S^\sigma.
\]
We can either apply $F(-, X)$ or $(-) \otimes X$ to this, and get
\begin{lemma}\label[lemma]{lemma:underlying-les}
  There are long exact sequences
  \[
    \cdots \longrightarrow \pi_{s + 1, w + 1} X \overset\rho\longrightarrow \pi_{s, w} X \overset{\mathrm{res}}\longrightarrow \pi_{s, w} (C_2)_+ \otimes X \longrightarrow \pi_{s, w + 1} X \overset\rho\longrightarrow \cdots.
  \]
  and
  \[
    \cdots \overset{\rho}\longrightarrow \pi_{s, w - 1} X \longrightarrow \pi_{s, w} (C_2)_+ \otimes X \overset{\mathrm{tr}}\longrightarrow \pi_{s, w} X \overset\rho\longrightarrow \pi_{s - 1, w - 1} X \longrightarrow \cdots .
  \]
  Moreover, $\mathrm{res} \circ \mathrm{tr}$ is $1 + \gamma$ on $\pi_* \iota X = \pi_{*, 0} (C_2)_+ \otimes X$, where $\gamma$ is the generator of $C_2$.
\end{lemma}

\begin{corollary}\label[corollary]{cor:underlying-inj}
  Suppose $\rho$ is injective on $\pi_{*, 0} X$ and $\pi_{*, 1} X$. Then
  \[
    \pi_* \iota X = (\pi_\star X / \rho)_{*, 0}
  \]
  with $\gamma = -1$.
\end{corollary}

\begin{proof}
  We first apply the first sequence with $w = 0$. By assumption, the last $\rho$ is injective, hence has trivial kernel. So $\mathrm{res}$ is surjective with kernel given by $\rho$-multiples.

  Now consider the second sequence with $w = 0$. Since the second $\rho$ is injective, it has trivial kernel, so $\mathrm{tr}$ has zero image. Hence $1 + \gamma = \mathrm{res} \circ \mathrm{tr} = 0$.
\end{proof}

\subsection{The \texorpdfstring{$C_2$}{C2}-equivariant Steenrod algebra}
We recall the computation of the $C_2$-equivariant Steenrod algebra.

\begin{lemma}[{\cite[Proposition 6.2]{hu-kriz}}]
  \[
    \Mt \equiv \pi_\star \HF_2 = \F_2[\tau, \rho] \oplus \left\{\frac{\gamma}{\rho^k \tau^\ell}\mid k, \ell \geq 1\right\}
  \]
  with the obvious ring structure. The degrees of the classes are
  \[
    |\tau| = (0, -1),\quad |\rho| = (-1, -1),\quad |\gamma| = (1, 0).
  \]
\end{lemma}
The terms involving $\gamma$ are known as the ``negative cone''. They will not play a role in the story (though we will have to verify that they indeed do not play a role).

\begin{lemma}[{\cite[Theorem 6.41]{hu-kriz}}]
  \[
    \A^{C_2} \equiv \pi_\star (\HF_2 \otimes \HF_2) = \Mt[\tau_0, \tau_1, \ldots, \xi_1, \xi_2 \ldots]/(\tau_i^2 = \rho \tau_{i + 1} + \bar\tau \xi_{i + 1}),
  \]
  where $\bar\tau = \tau + \rho \tau_0$ and
  \[
    |\tau_n| = (2^{n + 1} - 1, 2^n - 1),\quad |\xi_n| = (2^{n + 1} - 2, 2^n - 1).
  \]
  We have
  \[
    \eta_R(\rho) = \rho,\quad \eta_R(\tau) = \bar{\tau} = \tau + \rho \tau_0
  \]
  and
  \[
    \Delta \xi_k = \sum \xi_{k - i}^{2^i} \otimes \xi_i,\quad \Delta \tau_k = \sum_{i = 0}^k \xi_{k - i}^{2^i} \otimes \tau_i + \tau_k \otimes 1.
  \]
\end{lemma}

One observes that \Cref{cor:underlying-inj} applies to $\Mt$ and $\A^{C_2}$, so we can read off the homotopy groups of the underlying spectra of $\HF_2$ and $\HF_2 \otimes \HF_2$. This tells us $\iota \HF_2 = H\F_2$ and $\pi_* \iota (\HF_2 \otimes \HF_2) \cong \A$, as expected. The identification of the latter with $\A$ is canonical for the following unimaginative reason:

\begin{lemma}\label[lemma]{lemma:steenrod-auto}
  Let $\zeta_1, \zeta_2, \ldots \in \A$ be such that $|\zeta_i| = 2^i - 1$ and
  \[
    \Delta \zeta_k = \sum \zeta_{k - i}^{2^i} \otimes \zeta_i.
  \]
  If $\zeta_1 \not= 0$, then $\zeta_i = \xi_i$.
\end{lemma}

\begin{proof}
  We prove this by induction on $i$. It is clear for $i = 1$, since there is a unique element of degree $1$. Then if it is true for $k < i$, then $\zeta_i - \xi_i$ is primitive. But the only primitive elements of $\A$ are of the form $\xi_1^{2^k}$, which is not in the bidegree. So this is zero, and $\zeta_i = \xi_i$.
\end{proof}

\subsection{The homology of geometric fixed points}
Since $(-)^{\Phi C_2}$ is symmetric monoidal, in order to deduce $H_* X^{\Phi C_2}$ from $(\HF_2)_\star X$, it suffices to understand $\HF_2^{\Phi C_2}$.

\begin{lemma}
  We have $\HF_2^{\Phi C_2} = H\F_2 [\tau]$ as an $\mathbb{E}_1$-ring, where $|\tau| = 1$.
\end{lemma}

Truncation gives us an $\mathbb{E}_\infty$ map $\HF_2^{\Phi C_2} \to H\F_2$, which we can understand as quotienting out $\tau$.

\begin{proof}
  The inclusion of the categorical fixed points gives $\HF_2^{\Phi C_2}$ the structure of an $\mathbb{E}_\infty$-$H\F_2$-algebra. Since $H\F_2[\tau]$ is the free $\mathbb{E}_1$-ring on one generator over $H\F_2$, it suffices to observe that $\pi_* \HF_2^{\Phi C_2} = \F_2[\tau]$ by \Cref{lemma:geometric-homotopy}.
\end{proof}

\begin{corollary}
  If $X \in \Sp_{C_2}$, then
  \[
    H_* X^{\Phi C_2} \cong ((\HF_2)_\star X)/(\rho - 1, \tau).
  \]
  If $X$ is a homotopy ring, then this is an isomorphism of rings.
\end{corollary}

We are, of course, not only interested in the homology groups as groups, but as Steenrod modules. This involves understanding the comparison $\HF_2^{\Phi C_2} \otimes \HF_2^{\Phi C_2} \to H\F_2 \otimes H\F_2$ given by squaring the projection $\HF_2^{\Phi C_2} \to H\F_2$.

\begin{corollary}
  The map $\pi_* (\HF_2^{\Phi C_2} \otimes\HF_2^{\Phi C_2}) \to \pi_* (H\F_2 \otimes H\F_2)$ sends $\xi_i \mapsto \xi_i$ and kills $\tau_0, \tau$.
\end{corollary}

\begin{proof}
  After setting $\rho = 1$, we can express $\tau_{i + 1}$ in terms of $\tau_i^2$ and $\xi_{i + 1}$, so $\A^{C_2}/(\rho - 1)$ is generated polynomially over $\F_2[\tau]$ by $\tau_0, \xi_1, \xi_2, \ldots$ with no relations. We know that $\eta_L (\tau)$ and $\eta_R (\tau)$ get mapped to zero, and these are $\tau$ and $\tau + \tau_0$ respectively. So $\tau$ and $\tau_0$ get killed.

  As a spectrum, $\HF_2^{\Phi C_2}$ splits as a sum of $H\F_2$'s, we know that the map is surjective, and thus $\xi_1$ must hit $\xi_1$ since it is the only element left in the degree. Finally, the images of the $\xi_i$ satisfy the conditions of \Cref{lemma:steenrod-auto}, so must be equal to $\xi_i$.
\end{proof}

\begin{corollary}
  If $X \in \Sp_{C_2}$, then the coaction on $H_* X^{\Phi C_2}$ is given by taking the $C_2$-equivariant coaction
  \[
    (\HF_2)_\star X \to \A^{C_2} \otimes_{\Mt} (\HF_2)_\star X
  \]
  and quotienting out by $(\rho - 1, \tau)$ on the left and $(\rho - 1, \tau \otimes 1, \tau_0 \otimes 1)$ on the right.
\end{corollary}

\subsection{Spectra whose homology is \texorpdfstring{$\aesq{n} \Mt$}{}}

We now introduce the standard quotient coalgebras of $\A^{C_2}$:
\begin{definition}
  We define
  \[
    \A^{C_2}(n) = \A^{C_2}/(\xi_i^{2^{n - i + 1}}, \tau_{n + 1}, \tau_{n + 2}, \ldots).
  \]
  If $n < 0$, we set $\A^{C_2}(n) = \Mt$. Then we have
  \[
    \aesq{n} \Mt = \Mt[\bxi_1^{2^n}, \ldots, \bxi_n^2, \bxi_{n + 1}, \bxi_{n + 2}, \ldots, \bar{\tau}_{n + 1}, \bar{\tau}_{n + 2}, \ldots].
  \]
\end{definition}
As usual, the bar denotes the Hopf algebroid antipode.

\begin{definition}
  We let $\Y_n \in \Sp_{C_2}$ be any spectrum such that
  \[
    (\HF_2)_\star \Y_n = \A^{C_2} \square_{\A^{C_2}(n)} \Mt
  \]
  and the $C_2$ action on the underlying spectrum is trivial. If $n \geq m$, a map $f\colon \Y_n \to \Y_m$ is admissible if it induces the natural injection in homology.
\end{definition}

It is helpful to note that
\begin{lemma}
  There is a unique admissible map $\Y_n \to \HF_2$.
\end{lemma}
Since we understand the situation of $\HF_2$ completely, we get to learn about $\Y_n$ through comparison.

\begin{proof}
  Since $H_\star \Y_n$ is free over $\Mt$, the universal coefficients theorem tells us such maps are classified by $\Mt$-module maps $H_\star \Y_n \to \Mt$, which in turn is in bijection with $\A^{C_2}$-comodule maps $H_\star \Y_n \to \A^{C_2}$.
\end{proof}

\begin{lemma}
  $H_* \iota \Y_n = \A \square_{\A(n)} \F_2$ with trivial $C_2$ action as a Steenrod comodule, and admissible maps induce the natural inclusion.
\end{lemma}
This implies that $\Y_n$ cannot exist for $n > 2$.

\begin{proof}
  Since $\iota$ is symmetric monoidal, we have $H_* \iota \Y_n = \pi_* \iota (\HF_2 \otimes \Y_n)$.

  Observe that in $(\HF_2)_\star \Y_n$, the $\rho$-torsion term of smallest weight is $\frac{\gamma}{\rho \tau}$ with weight $2$. So $\rho$ is injective on $(\HF_2)_{*, 0} \Y_n$ and $(\HF_2)_{*, 1} \Y_n$, and \Cref{cor:underlying-inj} applies. This tells us $H_* \iota \Y_n = ((\HF_2)_\star \Y_n / \rho)_{*, 0}$. It is then easy to check that the ring map $\A \square_{\A(n)} \F_2 \to H_* \iota \Y_n$ sending $\bxi_i^k$ to $(\tau^{2^{i - 1} - 1}\bar{\tau}_{i - 1})^k$ is an isomorphism. Since the ring is $2$-torsion, we have $\gamma = -1 = 1$.

  The comodule structure follows from a similar calculation.
\end{proof}

We can similarly compute
\begin{lemma}
  \[
    H_* \Y_n^{\Phi C_2} = \F_2[\bar{\tau}_{n + 1}, \bxi_1^{2^n}, \ldots, \bxi_n^2, \bxi_{n + 1}, \ldots]
  \]
  with the coactions of the $\bxi_i$ as usual and
  \[
    \psi (\bar{\tau}_{n + 1}) = 1 \otimes \bar{\tau}_{n + 1}.
  \]
  An admissible map $Y_n \to Y_m$ acts as the ``identity'' on the $\bxi_i^{2^n}$'s and sends $\bar{\tau}_{n + 1}$ to $\bar{\tau}_{m + 1}^{2^{n - m}}$.
\end{lemma}

Our goal is to understand the homology of the categorical fixed points $\Y_n^{C_2}$. To do so, we make use of the commutative diagram
\[
  \begin{tikzcd}
    (\Y_n)_{h C_2} \ar[d, equals] \ar[r] & \Y_n^{C_2} \ar[r] \ar[d] & \Y_n^{\Phi C_2} \ar[d] \ar[r] & \Sigma (\Y_n)_{h C_2} \ar[d, equals]\\
    (\Y_n)_{h C_2} \ar[r] & \iota \Y_n \ar[r] & \widetilde{\Y_n} \ar[r] & \Sigma (\Y_n)_{h C_2},
  \end{tikzcd}
\]
where $\widetilde{Y_n}$ is defined to be the cofiber of $(Y_n)_{h C_2} \to \iota Y_n$.

To compute $H_* \Y_n^{C_2}$, we have to understand the effect of $\Y_n^{\Phi C_2} \to \Sigma (\Y_n)_{h C_2}$ on homology. This factors through $\widetilde{\Y_n}$, and the strategy is to understand each factor separately. By \Cref{thm:stunted}, we can identify the homology of the sequence
\[
  \iota \Y_n \to \widetilde{\Y_n} \to \Sigma (\Y_n)_{h C_2}
\]
with the natural maps
\[
  \A \square_{\A(n)} \F_2 \to \A \square_{\A(n)} H_* \Sigma \RP^\infty_{-1} \to \A \square_{\A(n)} H_* \Sigma \RP^\infty,
\]
which is a short exact sequence.

As for the map $\Y_n^{\Phi C_2} \to \widetilde{\Y_n}$, we know it is an equivalence in the case $n = -1$. In particular, it is injective in homology. Since $H_* \Y_n^{\Phi C_2}$ and $H_* \widetilde{\Y_n}$ inject into the $n = -1$ counterparts, we know that $H_* \Y_n^{\Phi C_2} \to H_* \widetilde{\Y_n}$ is injective as well. It remains to compute the image.

At this moment it is slightly more convenient to consider cohomology. We follow some constructions from \cite{lin-ext}.

For any $N$, we have a map $\Sigma \RP^\infty_{-N} \to \Sigma \RP^\infty_{-1}$ which presents $H^* \Sigma \RP^\infty_{-1}$ as a submodule of $H^* \Sigma \RP^\infty_{-N}$. Let $M_{n, N}^* \subseteq H^* \Sigma \RP^\infty_{-N}$ be the $\A(n)^*$-submodule generated under $\A(n)^*$ by elements of negative degree, and let $L_n^* = M_{n, N}^* \cap H^* \Sigma \RP^\infty_{-1}$ for sufficiently large $N$ (where it no longer depends on $N$). Let $G_n^* = H^* \Sigma \RP^\infty_{-1} / L_n^*$. We then have a short exact sequence of $\A(n)$-comodules.
\[
  0 \to G_n \to H_* \Sigma \RP^\infty_{-1} \to L_n \to 0.
\]

\begin{lemma}
  $L_n^*$ is trivial in degree $0$.
\end{lemma}

\begin{proof}
  In the notation of \cite[Page 1]{lin-ext}, we need to know that $x^{-1}$ is not in the image of the $\Sq^k$'s. We have
  \[
    \Sq^k x^{-k - 1} = \frac{(-k - 1)(-k - 2) \cdots (-2k)}{1 \cdot 2 \cdot\cdots \cdot k} x^{-1} = (-1)^{?} \binom{2k}{k} x^{-1} = 0.
  \]
\end{proof}

\begin{lemma}
  There are no non-zero $\A$-comodule maps $H_* \Y_n^{\Phi C_2} \to \A \square_{\A(n)} L_n$.
\end{lemma}

\begin{proof}
  By adjunction, this is equivalent to showing that there are no non-zero $\A(n)$-comodule maps $H_* \Y_n^{\Phi C_2} \to L_n$, or equivalently, no $\A(n)^*$-module maps $L_n^* \to H^* \Y_n^{\Phi C_2}$.

  Now note that in $L_n^*$ all elements have positive degree. Moreover, since $\A(n)^*$ is generated by elements of degree at most $2^n$, the module $L_n^*$ is generated as an $\A(n)^*$-module by elements of degree less than $2^n - 1$. On the other hand, the elements of lowest degree in $H_* \Y_n^{\Phi C_2}$ are in degree $0$ and $2^n$ (given by $1$ and $\bxi_1^{2^n}$). Thus, any $\A(n)^*$-module map $L_n \to H^* \Y_n^{\Phi C_2}$ must kill the generators, hence must be zero.
\end{proof}

\begin{corollary}
  The map $H_* \Y_n^{\Phi C_2} \to H_* \widetilde{\Y_n}$ factors through $\A \square_{\A(n)} G_n$.
\end{corollary}

\begin{corollary}
  $H_* \Y_n^{\Phi C_2} \to \A \square_{\A(n)} G_n$ is an isomorphism.
\end{corollary}

\begin{proof}
  Since $H_* \Y_n^{\Phi C_2} \to H_* \widetilde{\Y_n}$ is injective, so must this map. By \cite[Lemma 1.3]{lin-ext} both sides are abstractly isomorphic as graded groups. Since they are finite dimensional in each degree, it must be an isomorphism.
\end{proof}

Since $H_* \widetilde{\Y_n} \to H_* \Sigma (\Y_n)_{h C_2}$ is surjective in homology with kernel $\A \square_{\A(n)} \F_2$, we learn that:
\begin{corollary}
  There is a short exact sequence of comodules
  \[
    0 \to \A \square_{\A(n)} L_n[-1] \to H_* \Y_n^{C_2} \to \A \square_{\A(n)} \F_2 \to 0.
  \]
\end{corollary}

These $L_n$ can be explicitly computed.
\begin{example}
  $L_0[-1] = 0$. $L_1[-1] = \F_2$. $L_2[-1] = \F_2 \oplus H_* L$.
\end{example}

Specializing to the $n = 2$ case,
\begin{lemma}
  When $n = 2$, the short exact sequence
  \[
    0 \to \A \square_{\A(2)} L_2[-1] \to H_* \Y_2^{C_2} \to \A \square_{\A(2)} \F_2 \to 0.
  \]
  splits non-canonically.
\end{lemma}

\begin{proof}
  It is classified by an element in 
  \[
    \Ext_{\A}^{1, 0} (\A \square_{\A(2)} \F_2, \A \square_{\A(2)} L_2[-1]) = \Ext_{\A(2)} (\A \square_{\A(2)} \F_2, L_2[-1]).
  \]
  Thus, we have to equivalently show that any short exact sequence of $\A(2)^*$-modules
  \[
    0 \to \A^*/\!/\A(2)^* \to Z \to L_2[-1]^* \to 0
  \]
  splits. Let $x_i \in L_2[-1]^*$ be the unique element in degree $i$, if exists.

  Let $z_0 \in Z$ be any lift of $x_0 \in L_2[-1]^*$ and $z_1$ the unique lift of $x_1$. Then we can attempt to produce a splitting by mapping $x_0$ to $z_0$ and $x_1$ to $z_1$. Since the lowest non-zero degrees of $\A^*/\!/\A(2)^*$ are $0, 8, 12$, the only obstruction is if we cannot map $x_8$ to $\Sq^4 \Sq^2 \Sq^1 z_1$. This occurs if and only if $\Sq^4 \Sq^4 \Sq^2 \Sq^1 z_1 \not= 0$. But $\Sq^4 \Sq^4 \Sq^2 \Sq^1 = \Sq^7 \Sq^3 \Sq^1$, and $\Sq^3 \Sq^1 z_1$ has degree $5$, where $Z$ is trivial. So this cannot happen.
\end{proof}

\begin{corollary}
  We have
  \[
    H_* \Y_2^{C_2} = \A \square_{\A(2)} (\F_2 \oplus \F_2 \oplus H_* L).
  \]
\end{corollary}

\begin{corollary}
  If $\Y_2^{C_2}$ admits the structure of a $\tmf$-module, then upon 2-completion, we have an isomorphism
  \[
    \Y_2^{C_2} = \tmf \oplus \tmf \oplus \tmf \otimes L.
  \]
\end{corollary}

\begin{proof}
  The Adams spectral sequence has two $h_0$-towers at $0$, and this lets us map in two copies of $\tmf$. After quotienting them out we are left with $\Ext_{\A(2)}(\F_2, H_* L)$, which is shown in \Cref{figure:ext-l}.

  \begin{figure}[h]
    \centering
    \DeclareSseqGroup \tower {} {\class(0,0)\savestack\DoUntilOutOfBounds {\class(\lastx,\lasty+1)\structline}\restorestack}
    \DeclareSseqCommand\hone{ O{} d() }{\IfNoValueF{#2}{\pushstack(#2) }\class[#1] (\lastx+1, \lasty+1)\structline}
    \DeclareSseqCommand\htwo{ O{} d() }{\IfNoValueF{#2}{\pushstack(#2) }\class[#1] (\lastx + 3, \lasty+1)\structline}
    \DeclareSseqCommand\hook{ d() }{\tower(#1)\hone\hone}

    \begin{sseqpage}[classes = {fill, minimum size=0.07cm}, scale = {\ifspringer 0.406 \else 0.42\fi}, y tick step = 2, x tick step = 2, grid = go, x range={0}{25}, yrange={0}{10}, class pattern=linear]
      \class(1, 0)
      \htwo \htwo
      \hone(1, 0)
      \tower(4, 2)
      \hook(8, 3)
      \tower(12, 6)
      \hook(16, 3)
      \class(16, 3) \htwo \htwo
      \class(21, 4) \structline(22, 5)
      \hook(16, 7)
      \tower(20, 6) \tower(20, 10)
      \hook(24, 7) \hook(24, 11)
    \end{sseqpage}
    \caption{$\Ext_{\A(2)}(\F_2, H_* L)$}\label{figure:ext-l}
  \end{figure}

  Let $w_1$ be the class in degree $1$. To map in $L$, we need to show that
  \[
    2w_1 = 0, \quad \langle \eta, 2, w_1\rangle = 0,\quad \langle \nu, \eta, 2, w_1\rangle = 0.
  \]
  These obstructions live in $\pi_1, \pi_3$ and $\pi_7$ respectively. The only non-trivial case is the last case, but the only possible class lives in the indeterminacy. So we get a map $L \to Y_2^{C_2}$, which lifts to a map $\tmf \otimes L \to Y_2^{C_2}$ via the $\tmf$-module structure. This map induces an isomorphism on homology groups, hence an equivalence on $2$-completion.
\end{proof}

\section{Stunted projective spaces}\label{section:stunted}
The goal of this appendix is to prove the following folklore theorem we used in \Cref{thm:x11-permanent}:

\begin{theorem}\label{thm:stunted}
  Let $X$ be a genuine $C_2$-spectrum whose $C_2$ action on the underlying spectrum $\iota X$ is trivial. Then the cofiber of the composition
  \[
    \iota X \otimes \RP^n_+ \to \iota X \otimes \RP^\infty_+ = X_{h C_2} \overset{\mathrm{Nm}}\to X^{h C_2} = (\iota X)^{\RP^\infty_+} \to (\iota X)^{\RP^m_+}
  \]
  is $\iota X \otimes \Sigma P_{-m - 1}^n$, where $P_{-m}^n$ is the stunted projective space.
\end{theorem}

We should think of $\iota X \otimes \RP^n_+$ as the colimit of the constant functor on $\iota X$ under $\RP^n$. In general, we can understand this colimit as follows:

\begin{lemma}\label[lemma]{lemma:adams}
  Let $X$ be any genuine $C_2$-spectrum, and restrict it to a diagram on $\RP^n$ under the inclusion $\RP^n \hookrightarrow BC_2 = \RP^\infty$. Then
  \[
    \colim_{\RP^n} X \simeq (S((n + 1)\sigma)_+ \otimes X)_{h C_2} \simeq (S((n + 1)\sigma)_+ \otimes X)^{C_2},
  \]
  where if $V \in \RO(C_2)$, then $S(V)$ is the unit sphere of $V$.
\end{lemma}
Note that since $C_2$ acts freely on $S((n + 1) \sigma)$, this homotopy quotient agrees with the literal quotient under point set models.

\begin{proof}
  The idea of the first isomorphism is that we have a $2$-sheeted universal cover $S((n + 1)\sigma) \to \RP^n$ with deck transformation group $C_2$. To take the colimit over $\RP^n$, we can take the colimit over $S((n + 1)\sigma)$, and then take the colimit over the residual $C_2$ action. The restriction of $X$ to $S((n + 1) \sigma)$ is trivial since the following diagram commutes
  \[
    \begin{tikzcd}
      S((n + 1)\sigma) \ar[r] \ar[d] & S(\infty \sigma) = * \ar[d]\\
      \RP^n \ar[r] & \RP^\infty,
    \end{tikzcd}
  \]
  so the colimit over $S((n + 1) \sigma)$ is simply given by tensoring with $S((n + 1) \sigma)_+$.

  To prove this formally, we apply \cite[Corollary 4.2.3.10]{htt}, where we choose $\mathcal{J} = BC_2$, $K = \RP^n$ and $F$ to be the universal cover $S((n + 1)\sigma) \to \RP^n$ with $\mathcal{J}$ acting via deck transformations.

  The second equivalence is the Adams isomorphism (see e.g.\ \cite[Theorem 8.4]{equivariant-lectures}).
\end{proof}

We recall some basic properties of stunted projective spaces. Let $\phi(n)$ be the number of integers $j$ congruent to $0, 1, 2$ or $4$ mod $8$ such that $0 < j \leq n$.

\begin{theorem}[{\cite[Theorem V.2.14]{h-infinity}}]
  There is a (unique) collection of spectra $\{P_m^n\}_{n \geq m}$ with the property that
  \begin{enumerate}
    \item If $m > 0$, then $P^n_m = \Sigma^\infty \RP^n / \RP^{m - 1}$ using the standard inclusion.
    \item If $r \equiv n \pmod {2^{\phi(k)}}$, then $P_n^{n + k} \cong \Sigma^{n - r} P^{r + k}_r$.
  \end{enumerate}
  Further, we have
  \begin{enumerate}
    \setcounter{enumi}{2}
    \item If we restrict the $BC_2$ action on $S^{\sigma}$ to $\RP^m \subseteq \RP^\infty = BC_2$, then
      \[
        \colim_{\RP^m} S^{n\sigma} \simeq P_n^{n + m}
      \]
    \item $D P_m^n \simeq \Sigma P_{-n-1}^{-m-1}$.
  \end{enumerate}
\end{theorem}

\begin{corollary}
  There is a natural isomorphism
  \[
    \lim_{\RP^m} S^{n \sigma} \simeq \Sigma P_{n - m + 1}^{n - 1}.
  \]
\end{corollary}

\begin{proof}
  \[
    \lim_{\RP^m} S^{n \sigma} \simeq D \colim_{\RP^m} S^{-n \sigma} \simeq D P_{-n}^{m - n} \simeq \Sigma P_{n - m - 1}^{n - 1}.
  \]
\end{proof}

\begin{corollary}\label[corollary]{cor:sns}
  Let $X$ be a spectrum with trivial $C_2$ action. Then
  \[
    (S^{n \sigma} \otimes X)^{h C_2} \simeq \lim_N (\Sigma P^{n - 1}_{-N} \otimes X).
  \]
\end{corollary}

\begin{proof}
  We write $X$ as a filtered colimit of finite spectra, $X = \colim_\alpha X^{(\alpha)}$. Then we have
  \[
    \begin{aligned}
      (S^{n \sigma} \otimes X)^{h C_2} &= \left(S^{n \sigma} \otimes \colim_\alpha X^{(\alpha)}\right)^{h C_2} \\
                                      &= \left(\colim_\alpha \left(S^{n \sigma} \otimes X^{(\alpha)}\right)\right)^{h C_2} \\
                                      &= \lim_{\RP^\infty} \colim_\alpha \left(S^{n \sigma} \otimes X^{(\alpha)}\right) \\
                                      &= \lim_m \lim_{\RP^m} \colim_\alpha \left(S^{n \sigma} \otimes X^{(\alpha)}\right) \\
                                      &= \lim_m \colim_\alpha \lim_{\RP^m} \left(S^{n \sigma} \otimes X^{(\alpha)}\right) \\
                                      &= \lim_m \colim_\alpha \left(\Sigma P_{n - m - 1}^{n - 1} \otimes X^{(\alpha)}\right) \\
                                      &= \lim_m \left(\Sigma P_{n - m - 1}^{n - 1} \otimes X\right),
    \end{aligned}
  \]
  where we use that finite limits commute with filtered colimits and tensoring with \emph{finite} spectra.
\end{proof}

We now recall the construction of the norm map. Let $X$ be a genuine $C_2$-spectrum. By \Cref{lemma:adams}, we have an isomorphism
\[
  \colim_{\RP^n} X = (S((n + 1)\sigma)_+ \otimes X)^{C_2}.
\]
Taking the colimit as $n \to \infty$ gives
\[
  X_{hC_2} = (S(\infty \sigma)_+ \otimes X)^{C_2},
\]
where $S(\infty \sigma) = \colim_n S(n \sigma)$. The norm map $X_{h C_2} \to X^{C_2}$ is then induced by the projection $S(\infty \sigma)_+ \to S^0$. By replacing $X$ with the cofree version $X^{(EC_2)_+}$, we obtain the $X^{h C_2}$-valued version of the norm map.

\begin{lemma}
  Let $X$ be a genuine $C_2$-spectrum whose $C_2$ action on the underlying spectrum $\iota X$ is trivial. Then the cofiber of the composition
  \[
    \iota X \otimes \RP^n_+ \to \iota X \otimes \RP^\infty_+ = X_{h C_2} \overset{\mathrm{Nm}}\to X^{h C_2} = (\iota X)^{\RP^\infty_+}
  \]
  is $\lim_N (\iota X \otimes \Sigma P_{-N}^n)$.
\end{lemma}

\begin{proof}
  This composition is obtained by taking the fixed points of
  \[
    S((n + 1) \sigma)_+ \otimes X^{(EC_2)_{+}} \to X^{(EC_2)_+}.
  \]
  The cofiber of this map is $S^{(n+1) \sigma} \otimes X^{(EC_2)_+} \cong (S^{(n + 1)\sigma} \otimes X)^{(EC_2)_+}$. So this follows from \Cref{cor:sns}.
\end{proof}

We can now prove \Cref{thm:stunted}.
\begin{proof}[of \Cref{thm:stunted}]
  This follows from the commutative diagram
  \[
    \begin{tikzcd}
      & (\iota X)^{\RP^\infty_{m + 1}} \ar[r, equals] \ar[d] & \lim_N (\Sigma P^{-m - 2}_{-N} \otimes \iota X) \ar[d]\\
      \RP^n_+ \otimes \iota X \ar[r] \ar[d, equals] & (\iota X)^{\RP^\infty_+} \ar[r] \ar[d] & \lim_N (\Sigma P^{n}_{-N} \otimes \iota X) \ar[d] \\
      \RP^n_+ \otimes \iota X \ar[r] & (\iota X)^{\RP^m_+} \ar[r] & \Sigma P^n_{-m - 1} \otimes \iota X.
    \end{tikzcd}
  \]
\end{proof}
\section{Sage script}\label{section:sage}
This appendix contains the sage script used to perform the computer calculations. The actual computations are at the end of the script and the comments indicate the lemmas they prove.
\lstset{
  belowcaptionskip=1\baselineskip,
  language=Python,
  showstringspaces=false,
  basicstyle=\footnotesize\ttfamily,
  keywordstyle=\bfseries\color{green!50!black},
  commentstyle=\color{purple!40!black},
  identifierstyle=\color{blue},
  deletekeywords=[2]{reduce},
  escapeinside={@}{@},
}
\begin{lstlisting}
# Define the ring @$\Gamma [z]/(2z - a_1 z^2 + a_3 z^4)$@
# We first start with the generators
S.<a_1p, a_3p, sp, tp, zp> = PolynomialRing(QQ)

# Some convenient constants. The convention is that \psi(a) = a2
rp = (sp^2 + a_1p * sp)/3
a_12p = a_1p + 2 * sp
a_32p = a_3p + a_1p * rp + 2 * tp
a_32ip = a_32p^2 * (a_12p^3 - 27 * a_32p)
a_3ip = a_3p^2 * (a_1p^3 - 27 * a_3p)
dp = a_3p^3 * (a_1p^3 - 27 * a_3p)

s4 = 6 * sp * tp - a_1p * sp^3 + 3 * a_1p * tp + 3 * a_3p * sp
t2 = (sp^6 + 3 * a_1p * sp^5 - 9 * a_1p * sp^2 * tp
        + 3 * a_1p^2 * sp^4 - 9 * a_1p^2 * sp * tp
        + a_1p^3 * sp^3 - 27 * a_3p * tp) / 27

relations = (sp^4 - s4, tp^2 - t2,
        2 * zp - a_12p * zp^2 + a_32p * zp^4, dp - 1)

# Now we impose the relations
T.<a_1, a_3, s, t, z> = S.quo(relations)

a_12 = T(a_12p)
a_32 = T(a_32p)
a_32i = T(a_32ip)
a_3i = T(a_3ip)
r = T(rp)

# z2 is the coaction on z
z2 = (z - r * z^3) / (1 - s * z + t * z^3)

# Some sanity checks
assert a_32i * a_32 == 1
# This is the conjugate of the relation we imposed.
assert 2 * z2 - a_1 * z2^2 + a_3 * z2^4 == 0

# Computes \psi(x) - 1 (x) x
def diff(x):
    return T(x(a_1, a_3, s, t, z2) - x(a_12, a_32, s, t, z))

# Substitute occurrences of `old` with `new` once
def substitute(expr, old, new):
    quo, rem = expr.quo_rem(old)
    return quo * new + rem

# Keep performing substitutions until you can't
def substitute_full(expr, old, new):
    quo = 1
    while quo != 0:
        quo, rem = expr.quo_rem(old)
        expr = quo * new + rem
    return expr

# Substitutions used to simplify expressions
SUBS = [
    [zp^4, a_32ip * (-2 * zp + a_12p * zp^2)],
    [tp^2, t2],
    # We check this substitution below
    [sp^4 * tp, -a_1p^3*a_3p*sp - 4*a_1p^2*a_3p*sp^2
        - 3*a_1p*a_3p*sp^3 - a_1p^4*tp - 3*a_1p^3*sp*tp
        - 4*a_1p^2*sp^2*tp - 3*a_1p*sp^3*tp - 6*a_3p^2*sp
        + 3*a_1p*a_3p*tp - 3*a_3p*sp*tp],
    [sp^4, s4],
    [dp, 1],
]

# Check our substitutions
for old, new in SUBS:
    assert T(old - new) == 0

v = QQ.valuation(2)

# This function takes an expression, writes it in a canonical form,
# and then reduces mod 2^n
def reduce(result, n):
    original = result
    result = result.lift()

    # First impose @$a_3 z^4 - a_1 z^2 + 2z = 0$@ to get rid of all
    # higher multiples of z
    result = substitute_full(result, *(SUBS[0]))

    # In this case, quotienting as a module is the same as
    # quotienting as a ring.
    result = substitute(result, zp^3, a_32ip * (a_12p * zp - 2))

    # Then we get rid of all t^2 multiples. 
    result = substitute_full(result, *(SUBS[1]))

    # Next we get rid of all s^4 multiples. The expression for s^4
    # involves t, so we need to handle s^4 t separately to avoid
    # re-introducing t^2.
    while result.quo_rem(sp^4)[0] != 0:
        result = substitute_full(result, *(SUBS[2]))
        result = substitute(result, *(SUBS[3]))

    # Set @$\Delta = 1$@ and drop constants
    result = substitute_full(result, *(SUBS[4]))
    result = result.quo_rem(zp)[0] * zp

    # Finally, reduce mod 2^n
    final = 0
    for (c, f) in list(result):
        c = v.simplify(c, error=(n - 1), force=True)
        final += c * f
    return final

# Check comodule coactions (@\Cref{lemma:computer-action}@)
assert reduce(diff(a_3p^2*zp) - r^2*a_32*z^2, 1) == 0
assert reduce(diff(a_3p * zp^2), 1) == 0
assert diff(a_3p * zp^3 - a_1p * zp) == 0

# Check @$\sqrt{\Delta}$@ is invariant
assert reduce(diff(a_3p^2 * (1 + a_1p * zp)) * z, 1) == 0

# Check that z_{-8} exists (@\Cref{lemma:computer-z-8}@)
z4 = zp^2 - a_3ip * a_1p^2 * zp
z8 = a_3p^3 * (a_1p * z4 + 2 * zp)
assert reduce(diff(z8), 2) == 0

# Compute a Massey product (@\Cref{lemma:computer-massey}@)
assert reduce(diff(z4) + 3 * r * z8(a_12, a_32, s, t, z), 2) == 0
\end{lstlisting}

\ifspringer
  \bibliographystyle{spmpsci}
  \bibliography{tmfc2.bib}
\else
  \printbibliography
\fi
\end{document}